\newif\ifen\entrue

\documentclass[11pt]{article} 
\usepackage{kotex}
\usepackage{geometry}       
\usepackage{setspace}       
\usepackage{hyperref}       \hyperbaseurl{}    
\usepackage[ruled,linesnumbered]{algorithm2e}   \SetKwInOut{KwParams}{Parameters}
                                                \newcommand{\lIfElse}[3]{\lIf{#1}{#2 \textbf{else}~#3}}
\usepackage{stackengine}    \newcommand\xrowht[2][0]{\addstackgap[.5\dimexpr#2\relax]{\vphantom{#1}}}
\usepackage{fancyhdr}
\usepackage{graphicx}

\usepackage{authblk}        
\ifen
    \newcommand{\thetitle}{The College Application Problem}
    
    \newcommand{\thedate}{\today}
    \newcommand{\theauthor}{Max Kapur}
    \newcommand{\theotherauthor}{Sung-Pil Hong}
\else
    \newcommand{\thetitle}{대학 지원 최적화 문제}
    
    \newcommand{\thedate}{\today}
    \newcommand{\theauthor}{Max Kapur}
    \newcommand{\theotherauthor}{홍성필}
\fi

\title{\thetitle}
\author{\theauthor}
\author{\theotherauthor}
\date{\thedate}

\ifen
\affil{Seoul National University\\
Department of Industrial Engineering\\
Management Science/Optimization Lab\\
Correspondence: \url{maxkapur@snu.ac.kr}}
\else
\affil{서울대학교\\
산업공학과 \\
경영과학/최적화 연구실\\
교신: \url{maxkapur@gmail.com}}
\fi

\usepackage{amssymb}
\usepackage{amsmath}
\usepackage{amsthm}

\DeclareMathOperator*{\argmax}{arg\,max}

\ifen
    \newtheorem{theorem}{Theorem}
    \newtheorem{lemma}{Lemma}
    \newtheorem{corollary}{Corollary}
    \newtheorem{proposition}{Proposition}
    \theoremstyle{definition}
    \newtheorem{example}{Example}
    \newtheorem{definition}{Definition}
    \newtheorem{problem}{Problem}
    \newtheorem{assumption}{Assumption}
    \newtheorem{conjecture}{Conjecture}
\else
    \newtheorem{theorem}{정리}
    \newtheorem{lemma}{기본 정리}
    
    \newtheorem{proposition}{제의}
    \theoremstyle{definition}
    \newtheorem{example}{예}
    \newtheorem{definition}{정의}
    \newtheorem{problem}{문제}

\fi

\ifen\else

\fi

\begin{document}

\maketitle

\begin{abstract}
This paper considers the maximization of the expected maximum value of a portfolio of random variables subject to a budget constraint. We refer to this as the optimal college application problem. When each variable's cost, or each college's application fee, is identical, we show that the optimal portfolios are nested in the budget constraint, yielding an exact polynomial-time algorithm. When colleges differ in their application fees, we show that the problem is NP-complete. We provide four algorithms for this more general setup: a branch-and-bound routine, a dynamic program that produces an exact solution in pseudopolynomial time, a different dynamic program that yields a fully polynomial-time approximation scheme, and a simulated-annealing heuristic. Numerical experiments demonstrate the algorithms' accuracy and efficiency. 


\end{abstract}

%

\pagebreak

\tableofcontents

\pagebreak
\ifen \section{Introduction}  \fi
\ifen This paper considers the following optimization problem:
\fi
\begin{align} \label{headlineproblem}
\begin{split}
\text{maximize}\quad & v(\mathcal{X}) =  \operatorname{E}\Bigl[\max\bigr\{t_0,
\max\{t_j Z_j : j \in \mathcal{X}\}\bigr\}\Bigr] \\
\text{subject to}\quad & \mathcal{X} \subseteq \mathcal{C}, ~~\sum_{j\in \mathcal{X}} g_j \leq H
\end{split}
\end{align}
\ifen Here $\mathcal{C} = \{ 1 \dots m\}$ is an index set; $H > 0$ is a budget parameter; for $j = 1 \dots m$, $g_j > 0$ is a cost parameter and $Z_j$ is a random, independent Bernoulli variable with probability $f_j$; and for $j = 0\dots m$, $t_j\geq 0$ is a utility parameter. 
\fi

\ifen
We refer to this problem as the \emph{optimal college application} problem, as follows. Consider an admissions market with $m$ colleges. The $j$th college is named $c_j$. Consider a single prospective student in this market, and let each $t_j$-value indicate the utility she associates with attending $c_j$, where her utility is $t_0$ if she does not attend college. Let $g_j$ denote the application fee for $c_j$ and $H$ the student's total budget to spend on application fees. Lastly, let $f_j$ denote the student's probability of being admitted to $c_j$ if she applies, so that $Z_j$ equals one if she is admitted and zero if not. It is appropriate to assume that the $Z_j$ are statistically independent as long as $f_j$ are probabilities estimated specifically for this student (as opposed to generic acceptance rates). Then the student's objective is to maximize the expected utility associated with the best school she gets into within this budget. Therefore, her optimal college application strategy is given by the solution $\mathcal{X}$ to the problem above, where $\mathcal{X}$ represents the set of schools to which she applies. 
\fi

\ifen
The college application problem is not solely of theoretical interest. Due to the widespread perception that career earnings depend on college admissions outcomes, the solution of \eqref{headlineproblem} holds monetary value. In the American college consulting industry, students pay private consultants an average of \$200 per hour for assistance in preparing application materials, estimating their admissions odds, and identifying target schools (Sklarow 2018). In Korea, an important revenue stream for admissions consulting firms such as Megastudy (\url{megastudy.net}) and Jinhak (\url{jinhak.com}) is ``mock application'' software that claims to use artificial intelligence to optimize the client's application strategy. However, although predicting admissions outcomes on a school-by-school basis is a standard benchmark for logistic regression models (Acharya et al. 2019), we believe our study is the first to focus on the overall application strategy as an optimization problem.

The problem is also conformable to other competitive matching games such as job application. Here, the budget constraint may represent the time needed to complete each application, or a legal limit on the number of applications permitted.
\fi

\ifen \subsection{Methodological orientation} \fi
\ifen
The college application problem straddles several methodological universes. Its stochastic nature recalls classical portfolio allocation models. However, the knapsack constraint renders the problem NP-complete, and necessitates combinatorial solution techniques. We also observe that the objective function is also a submodular set function, although our approximation results suggest that college application is a relatively easy instance of submodular maximization.
\fi

\ifen
A problem similar to \eqref{headlineproblem} arose in an equilibrium analysis of the American college market by Fu (2014), who described college application as a ``nontrivial portfolio problem'' (226). In computational finance, traditional portfolio allocation models weigh the sum of expected profit across all assets against a risk term, yielding a concave maximization problem with linear constraints (Markowitz 1952; Meucci 2005). But college applicants maximize the expected value of their \emph{best} asset: If a student is admitted to her $j$th choice, then she is indifferent as to whether she gets into her $(j+1)$th choice. As a result, student utility is \emph{convex} in the utility associated with individual applications. Risk management is implicit in the college application problem because, in a typical admissions market, college preferability correlates inversely with competitiveness. That is, students negotiate a tradeoff between attractive, selective “reach schools” and less preferable “safety schools” where admission is a safer bet (Jeon 2015). Finally, the combinatorial nature of the college application problem makes it difficult to solve using the gradient-based techniques associated with continuous portfolio optimization. Fu estimated her equilibrium model (which considers application as a \emph{cost} rather than a constraint) by clustering the schools so that $m=8$, a scale at which enumeration is tractable. We pursue a more general solution.
\fi

\ifen
The integer formulation of the college application problem can be viewed as a kind of binary knapsack problem with a polynomial objective function of degree $m$. Our branch-and-bound and dynamic programming algorithms closely resemble existing algorithms for knapsack problems (Martello and Toth 1990, \S\,2.5--6). In fact, by manipulating the admissions probabilities, the objective function can be made to approximate a linear function of the characteristic vector to an arbitrary degree of accuracy, a fact that we exploit in our NP-completeness proof. Previous research has introduced various forms of stochasticity to the knapsack problem, including variants in which each item's utility takes a known probability distribution (Steinberg and Parks 1979; Carraway et al. 1993) and an online context in which the weight of each item is observed after insertion into the knapsack (Dean et al. 2008). Our problem superficially resembles the preference-order knapsack problem considered by Steinberg and Parks and Carraway et al., but these models lack the college application problem's singular ``maximax'' form. Additionally, unlike those models, we do not attempt to replace the real-valued objective function with a preference order over \emph{outcome distributions,} which introduces technical issues concerning competing notions of stochastic dominance (Sniedovich 1980). We take for granted the student's preferences over \emph{outcomes} (as encoded in the $t_j$-values), and focus instead on an efficient computational approach to the well-defined problem above.
\fi

\ifen
We take special interest in the validity of greedy optimization algorithms, such as the algorithm that iteratively adds the school that elicits the greatest increase in the objective function until the budget is exhausted. Greedy algorithms produce a \emph{nested} family of solutions parameterized by the budget $H$: If $H \leq H'$, then the greedy solution for budget $H$ is a subset of the greedy solution for budget $H'$. As Rozanov and Tamir (2020) remark, the knowledge that the optima are nested aids not only in computing the optimal solution, but in the implementation thereof under uncertain information. For example, in the United States, many college applications are due at the beginning of November, and it is typical for students to begin working on their applications during the prior summer because colleges reward students who tailor their essays to the target school. However, students may not know how many schools they can afford to apply to until late October. The nestedness property---or equivalently, the validity of a greedy algorithm---implies that even in the absence of complete budget information, students can begin to carry out the optimal application strategy by writing essays for schools in the order that they enter the optimal portfolio.
\fi

\ifen
For certain classes of optimization problems, such as maximizing a submodular set function over a cardinality constraint, a greedy algorithm is known to be a good approximate solution and exact under certain additional assumptions (Fisher et al. 1978). For other problems, notably the binary knapsack problem, the most intuitive greedy algorithm can be made to perform arbitrarily poorly (Vazirani 2001). We show results for the college application problem that mirror those for the knapsack problem: When each $g_j = 1$, the optimal portfolios are nested. This special case mirrors the centralized college application process in Korea, where there is no application fee, but students are allowed to apply to only three schools during the main admissions cycle. Unfortunately, the nestedness property does not hold in the general case, nor does the greedy algorithm offer any performance guarantee. However, we identify a fully polynomial-time approximation scheme (FPTAS) based on fixed-point arithmetic. 
\fi

\ifen
Finally, we remark that the objective function of \eqref{headlineproblem} is a nondecreasing submodular function. However, our research employs more elementary analytical techniques, and our approximation results are tighter than those associated with generic submodular maximization algorithms. For example, a well-known result of Fisher et al. (1978) implies that the greedy algorithm is asymptotically $(1 - 1/e)$-optimal for the $g_j = 1$ case of the college application problem, whereas we show that the same algorithm is exact. As for the general problem, an equivalent approximation ratio is achievable when maximizing a submodular set function over a knapsack constraint using a variety of relaxation techniques (Badanidiyuru and Vondrák 2014; Kulik et al. 2013). But in the college application problem, the existence of the FPTAS supersedes these results.
\fi

\ifen \subsection{Structure of this paper} \fi
\ifen Section \ref{preliminaries} introduces some additional notation and assumptions that can be imposed without loss of generality. We also introduce a useful variable-elimination technique and prove that the objective function is submodular. 

In Section \ref{homogappcosts}, we consider the special case where each $g_j = 1$ and $H$ is an integer $h \leq m$.  We show that an intuitive heuristic is in fact a $1/h$-approximation algorithm. Then, we show that the optimal portfolios are nested in the budget constraint, which yields an exact algorithm that runs in $O(hm)$-time. 

In Section \ref{hetappcosts}, we turn to the scenario in which colleges differ in their application fees. We show that the decision form of the portfolio optimization problem is NP-complete through a polynomial reduction from the binary knapsack problem. We provide four algorithms for this more general setup. The first is a branch-and-bound routine. The second is a dynamic program that iterates on total expenditures and produces an exact solution in pseudopolynomial time, namely $O(Hm + m \log m)$. The third is a different dynamic program that iterates on truncated portfolio valuations. It yields a fully polynomial-time approximation scheme that produces a $(1 - \varepsilon)$-optimal solution in $O(m^3 / \varepsilon)$ time.  The fourth is a simulated-annealing heuristic algorithm that demonstrates strong performance in our synthetic instances. 

In Section \ref{numericalexperiments}, we present the results of computational experiments that confirm the validity and time complexity results established in the previous two sections. We also investigate the empirical accuracy of the simulated-annealing heuristic.

In the conclusion, we identify possible improvements to our solution algorithms and introduce a few extensions of the model that capture the features of real-world admissions processes.
\fi

\ifen \section{Notation and preliminary results} \fi\label{preliminaries}
\ifen Before discussing the solution algorithms, we will introduce additional notation and a few preliminary results. For the remainder of the paper, unless otherwise noted, we assume with trivial loss of generality that each $g_j \leq H$, $\sum g_j > H$, each $f_j \in (0, 1]$, and $t_0 < t_1 \leq \cdots \leq t_m$. 
\fi

\ifen \subsection{Refining the objective function} \fi

\ifen
We refer to the set $\mathcal{X} \subseteq \mathcal{C}$ of schools to which a student applies as her \emph{application portfolio.} The expected utility the student receives from $\mathcal{X}$ is called its \emph{valuation}. 
\fi 
\begin{definition}[\ifen Portfolio valuation function\fi]
$v(\mathcal{X}) =  \operatorname{E}\left[\max\bigr\{t_0,
\max\{t_j Z_j : j \in \mathcal{X}\}\bigr\}\right]$.
\end{definition}
\ifen
\noindent It is helpful to define the random variable $X  = \max\{ t_j Z_j : j \in \mathcal{X}\}$ as the utility achieved by the schools in the portfolio, so that when $t_0 = 0$, $v(\mathcal{X}) = \operatorname{E}[X]$. Similar pairs of variables with script and italic names such as $\mathcal{Y}_h$ and $Y_h$ will carry an analogous meaning.
\fi

\ifen
Given an application portfolio, let $p_j(\mathcal{X})$ denote the probability that the student attends $c_j$. This occurs if and only if she \emph{applies} to $c_j$, is \emph{admitted} to $c_j$, and is \emph{rejected} from any school she prefers to $c_j$; that is, any school with higher index. Hence, for $j= 0\dots m$,
\fi
\begin{align}
p_j(\mathcal{X}) &= 
\begin{cases}
\displaystyle f_j  \prod_{\substack{i \in \mathcal{X}: \\ i > j}} (1 - f_{i}), \quad & j \in \{0\}\cup\mathcal{X}\\
0, \quad & \text{\ifen otherwise\fi}
\end{cases} 
\end{align}
\ifen
where the empty product equals one and $f_0= 1$. The following proposition follows immediately.
\fi

\begin{proposition}[\ifen Closed form of portfolio valuation function\fi]
\begin{align}
v(\mathcal{X}) &= \sum_{j=0}^m t_j p_j(\mathcal{X}) = \sum_{j\in\{0\}\cup\mathcal{X}} \Bigl( f_j t_j \prod_{\substack{i \in \mathcal{X}: \\ i > j}} (1 - f_{i}) \Bigr)  \label{closedformportfoliovaluationX}
\end{align}
\end{proposition}

\ifen
Next, we show that without loss of generality, we may assume that $t_0 = 0$.
\fi
\begin{lemma} \label{assumetzerozero}
\ifen 
For some $\gamma \leq t_0$, let $\bar t_j = t_j - \gamma$ for $j = 0 \dots m$. Then $v(\mathcal{X}; \bar t_j) = v(\mathcal{X};  t_j) -  \gamma$ regardless of $\mathcal{X}$. 
\fi
\end{lemma}
\begin{proof}
\ifen
By definition, $\sum_{j=0}^m p_j(\mathcal{X}) = \sum_{j \in \{0\}\cup\mathcal{X}} p_j(\mathcal{X}) = 1$. Therefore
\fi
\begin{align}
\begin{split}
v(\mathcal{X}; \bar t_j) &= \sum_{j\in \{0\}\cup\mathcal{X}}  \bar t_j p_j(\mathcal{X})
=\sum_{j\in \{0\}\cup\mathcal{X}} (t_j - \gamma) p_j(\mathcal{X}) \\
&=\sum_{j\in \{0\}\cup\mathcal{X}} t_j p_j(\mathcal{X})  - \gamma 
= v(\mathcal{X}; t_j) - \gamma
\end{split} 
\end{align}
\ifen which completes the proof.\fi
\end{proof}

\ifen \subsection{An elimination technique} \fi
\ifen
Now, we present a variable-elimination technique that will prove useful throughout the paper.\footnote{We thank Yim Seho for pointing out this useful transformation.} Suppose that the student has already resolved to apply to $c_k$, and the remainder of her decision consists of determining which \emph{other} schools to apply to. Writing her total application portfolio as $\mathcal{X} = \mathcal{Y} \cup \{k\}$, we devise a function $w(\mathcal{Y})$ that orders portfolios according to how well they ``complement'' the singleton portfolio $\{k\}$. Specifically, the difference between $v(\mathcal{Y} \cup\{k\})$ and $w(\mathcal{Y})$ is the constant $f_k t_k$.
\fi 

\ifen
To construct $w(\mathcal{Y})$, let $\tilde t_j$ denote the expected utility the student receives from school $c_j$ \emph{given} that she has been admitted to $c_j$ and applied to $c_k$. For $j < k$ (including $j = 0$), this is $\tilde t_j = (1- f_k) t_j + f_k t_k$; for $j > k $, this is $\tilde t_j = t_j$. This means that 
\begin{equation}\label{Vyastildet}
v(\mathcal{Y}\cup\{k\}) = \sum_{j \in \{0\} \cup \mathcal{Y}} \tilde t_j p_j(\mathcal{Y}).\end{equation}
The transformation to $\tilde t$ does not change the order of the $t_j$-values. Therefore, the expression on the right side of \eqref{Vyastildet} is itself a portfolio valuation function. In the corresponding market, $t$ is replaced by $\tilde t$ and $\mathcal{C}$ is replaced by $\mathcal{C}\setminus\{k\}$. To restore our convention that $t_0 = 0$, we obtain $w(\mathcal{Y})$ by taking $\bar t_j = \tilde t_j - \tilde t_0$ for all $j \neq k$ and letting
\begin{equation}  \label{wYvXminusconst}
w(\mathcal{Y})
= \sum_{j \in \{0\} \cup \mathcal{Y}} \bar t_j p_j(\mathcal{Y})
= \sum_{j \in \{0\} \cup \mathcal{Y}} \tilde t_j p_j(\mathcal{Y})- \tilde t_0
= v(\mathcal{Y}\cup\{k\}) -  f_k t_k \end{equation}
where the second equality follows from Lemma \ref{assumetzerozero}. The validity of this transformation is summarized in the following theorem, where we write $v(\mathcal{X}; \bar t)$ instead of $w(\mathcal{Y})$ to emphasize that $w(\mathcal{Y})$ is, in form, a portfolio valuation function. 
\fi

\begin{lemma}[\ifen Eliminate $c_k$\fi] \label{eliminationtheorem}
\ifen For $\mathcal{X} \subseteq \mathcal{C} \setminus \{k\}$, $v(\mathcal{X}\cup\{k\}; t)  = v(\mathcal{X}; \bar t) + f_k t_k$, where
\fi
\begin{align}\label{howtotransformtj}
\bar t_j = 
\begin{cases}
(1 - f_k) t_j, \quad & t_j \leq t_k \\
t_j - f_k t_k, \quad& t_j > t_k.
\end{cases}
\end{align}
\end{lemma}

\begin{proof}
\ifen It is easy to verify that \eqref{howtotransformtj} is the composition of the two transformations (from $t$ to $\tilde t$, and from $\tilde t$ to $\bar t$) discussed above.
\fi
\end{proof}

\ifen
\noindent The transformation \eqref{howtotransformtj} can be applied iteratively to accommodate the case where the student has already resolved to apply to multiple schools.
\fi

\ifen \subsection{Submodularity of the objective} \fi
\ifen 
Now, we show that the portfolio valuation function is submodular. This result is primarily of taxonomical interest and may be safely skipped. Our solution algorithms for the college application problem are better than generic algorithms for submodular maximization, and we prove their validity using elementary analysis.
\fi

\ifen
\begin{definition}[Submodular set function]
Given a ground set $\mathcal{C}$ and function $v : 2^{\mathcal{C}} \mapsto \mathbb{R}$, $v(\mathcal{X})$ is called a \emph{submodular set function} if and only if $v(\mathcal{X}) + v(\mathcal{Y}) \geq v(\mathcal{X}\cup\mathcal{Y}) + v(\mathcal{X}\cap\mathcal{Y})$
for all $\mathcal{X}, \mathcal{Y} \subseteq \mathcal{C}$. Furthermore, if $ v(\mathcal{X}\cup\{k\}) - v(\mathcal{X}) \geq 0$ for all $\mathcal{X} \subset \mathcal{C}$ and $k \in \mathcal{C} \setminus \mathcal{X}$, $v(\mathcal{X})$ is said to be a \emph{nondecreasing} submodular set function.
\end{definition}
\fi

\ifen
\begin{theorem}
The college application portfolio valuation function
$v(\mathcal{X})$ 
is a nondecreasing submodular set function.
\end{theorem}
\fi

\ifen
\begin{proof}
It is self-evident that $v(\mathcal{X})$ is nondecreasing. To establish its submodularity, we apply proposition 2.1.iii of Nemhauser and Wolsey (1978) and show that
\begin{equation}\label{nemhauseriii}
v(\mathcal{X} \cup \{j\}) - v(\mathcal{X}) \geq 
v(\mathcal{X} \cup \{j, k\}) - v(\mathcal{X} \cup \{k\})
\end{equation}
for $\mathcal{X} \subset \mathcal{C}$ and $j \neq k \in \mathcal{C} \setminus \mathcal{X}$. By Lemma \ref{eliminationtheorem}, we can repeatedly eliminate the schools in $\mathcal{X}$ according to \eqref{howtotransformtj} to obtain a portfolio valuation function $w(\mathcal{Y})$ with parameter $\bar t$ such that $w(\mathcal{Y}) = v(\mathcal{X} \cup \mathcal{Y}) + \text{const.}$ for any $\mathcal{Y} \subseteq \mathcal{C} \setminus \mathcal{X}$. Therefore, \eqref{nemhauseriii} is equivalent to
\begin{align}
& w(\{j\}) - w(\varnothing) \geq w(\{j, k\}) - w(\{k\}) \\
\iff \qquad &w(\{j\})  +  w(\{k\})  \geq w(\{j, k\})  \\
\iff \qquad &\operatorname{E}[\,\bar t_j Z_j\,] + \operatorname{E}[\,\bar t_k Z_k\,] 
\geq \operatorname{E}\bigl[\max\{ \bar t_j Z_j, \bar t_k Z_k \} \bigr]
\end{align}
which is plainly true. 
\end{proof}
\fi

\ifen \section{Homogeneous application costs}  \fi \label{homogappcosts}
\ifen In this section, we focus on the special case in which each $g_j = 1$ and $H$ is a natural number $h \leq m$.   We show that an intuitive heuristic is in fact a $1/h$-approximation algorithm, then derive an exact polynomial-time solution algorithm. 
Applying Lemma \ref{assumetzerozero}, we assume that $t_0 = 0$ unless otherwise noted. Throughout this section, we will call the applicant Alma, and refer to the corresponding optimization problem as Alma's problem. 
\fi

\begin{problem}[\ifen Alma’s problem\fi]
\ifen Alma's optimal college application portfolio is given by the solution to the following combinatorial optimization problem:
\fi
\begin{align}
\begin{split}
\text{maximize}\quad &  v(\mathcal{X}) = \sum_{j\in
\mathcal{X}} \Bigl( f_j t_j \prod_{\substack{i \in \mathcal{X}: \\ i > j}} (1 - f_{i}) \Bigr)\\
\text{subject to}\quad & \mathcal{X}\subseteq\mathcal{C}, \quad|\mathcal{X}| \leq h 
\end{split}
\end{align}
\end{problem}

\ifen \subsection{Approximation properties of a na\"ive solution}  \fi
\ifen
The expected utility associated with a single school $c_j$ is simply $\operatorname{E}[t_j Z_j] = f_j t_j$. It is therefore tempting to adopt the following strategy, which turns out to be suboptimal.
\fi
\begin{definition}[\ifen Na\"ive algorithm for Alma’s problem\fi] \label{naivealgorithm}
\ifen 
Apply to the $h$ schools having the highest expected utility $f_j t_j$.
\fi
\end{definition}
\ifen 
\noindent This algorithm's computation time is $O(m)$ using the PICK algorithm of Blum et al. (1973).

The basic error of the na\"ive algorithm is that it maximizes $\operatorname{E}\left[\,\sum t_j Z_j\, \right]$ instead of $\operatorname{E}\left[\max \{t_j Z_j\} \right]$. The latter is what Alma is truly concerned with, since in the end she can attend only one school. The following example shows that the na\"ive algorithm can produce a suboptimal solution.
\fi
\begin{example}
\ifen 
Suppose $m=3$, $h=2$,  $ t= (70, 80, 90) $, and $f = (0.4, 0.4, 0.3)$. 
Then the na\"ive algorithm picks $\mathcal{T} = \{1, 2\}$ with 
$v(\mathcal{T}) = 70(0.4)(1-0.4) + 80(0.4) = 48.8$.
But $\mathcal{X} = \{2, 3\}$ with
$v(\mathcal{X}) = 80(0.4)(1-0.3) + 90(0.3) = 49.4$
is the optimal solution. 
\fi
\end{example}
\ifen
\noindent In fact, the na\"ive algorithm is a $(1/h)$-approximation algorithm for Alma’s problem, as expressed in the following theorem.
\fi

\begin{theorem}[\ifen Accuracy of the na\"ive algorithm\fi] \label{oneoverhopt}
\ifen 
When the application limit is $h$, let $\mathcal{X}_h$ denote the optimal portfolio, and $\mathcal{T}_h$ the set of the $h$ schools having the largest values of $f_j t_j$. Then $v(\mathcal{T}_h) / v(\mathcal{X}_h) \geq 1/h$. 
\fi
\end{theorem}
\begin{proof}
\ifen
Because $\mathcal{T}_h$ maximizes the quantity $\operatorname{E}\bigl[ \sum_{j \in \mathcal{T}_h}\{ t_j Z_j \}\bigr]$, we have
\fi
\begin{align} \label{oneoverhopt}
\begin{split}
v(\mathcal{X}_h) &= \operatorname{E}\Bigl[ \max_{j \in \mathcal{X}_h}\{ t_j Z_j \}\Bigr] \leq \operatorname{E}\Bigl[ \sum_{j \in \mathcal{X}_h}\{ t_j Z_j \}\Bigr] \leq \operatorname{E}\Bigl[ \sum_{j \in \mathcal{T}_h}\{ t_j Z_j \}\Bigr] \\
&= h  \operatorname{E}\Bigl[ \tfrac{1}{h} \sum_{j \in \mathcal{T}_h}\{ t_j Z_j \}\Bigr]
\leq h  \operatorname{E}\Bigl[ \max_{j \in \mathcal{T}_h}\{ t_j Z_j \}\Bigr]
= h v(\mathcal{T}_h).
\end{split}
\end{align}
\ifen 
where the final inequality follows from the concavity of the $\max\{\}$ operator.
\fi
\end{proof}
\ifen
\noindent The following example establishes the tightness of the approximation factor. 
\fi

\begin{example} \label{tightexampleforoneoverhopt}
\ifen
Pick any $h$ and let $m = 2h$. For a small constant $\varepsilon \in (0, 1)$, define the market as follows.
\fi
\begin{center}
\begin{tabular}{r|cccccccc}
$j$   & $1$      & $\cdots$ & $h$   &$h+1$         &  $h+2$ & $\cdots$ &      $m-1$  & $m$            \\ \hline
$f_j$ & $1$     &  $\cdots$ & $1$      & $\varepsilon^{1}$ & $\varepsilon^{2}$ & $\cdots$ & $\varepsilon^{h-1}$ & $\varepsilon^{h}$ \\
$t_j$ & $1$      &  $\cdots$ & $1$      & $\varepsilon^{-1}$ & $\varepsilon^{-2}$ & $\cdots$ & $\varepsilon^{-(h-1)}$ & $\varepsilon^{-h}$
\end{tabular}%
\end{center}
\ifen
Since all $f_j t_j = 1$, the na\"ive algorithm can choose $\mathcal{T}_h = \{1, \dots, h\}$, with $v(\mathcal{T}_h) = 1$. But the optimal solution is $\mathcal{X}_h = \{h+1, \dots, m\}$, with
\fi
\begin{equation}
v(\mathcal{X}_h) = \sum_{j= h +1}^m \Bigl( f_j t_j \prod_{j' = j+1}^m (1 - f_{j'}) \Bigr) =  \sum_{j= 1}^h  (1 - \varepsilon)^{j} \approx h\ifen.\fi
\end{equation}
\ifen
Thus, as $\varepsilon$ approaches zero, we have $v(\mathcal{T}_h) / v(\mathcal{X}_h) \to 1/h$. (The optimality of $\mathcal{X}_h$ follows from the fact that it achieves the upper bound of Theorem \ref{oneoverhopt}.)
\fi
\end{example}


\ifen
Although the na\"ive algorithm is suboptimal, we can still find the optimal solution in $O(hm)$-time, as we will now show.
\fi

\ifen \subsection{The nestedness property}  \fi
\ifen It turns out that the solution to Alma's problem possesses a special structure: An optimal portfolio of size $h+1$ includes an optimal portfolio of size $h$ as a subset.
\fi

%

\begin{theorem}[\ifen Nestedness of optimal application portfolios\fi] \label{nestedapplication}
\ifen There exists a sequence of portfolios $\{\mathcal{X}_h\}_{h=1}^m$ satisfying the nestedness relation
\fi
\begin{equation}
\mathcal{X}_1 \subset \mathcal{X}_2\subset \dots \subset \mathcal{X}_m
\end{equation}
\ifen such that each $\mathcal{X}_h$ is an optimal application portfolio when the application limit is $h$.\fi
\end{theorem}

\begin{proof}\ifen
By induction on $h$. Applying Lemma \ref{assumetzerozero}, we assume that $t_0 = 0$. 

(Base case.) First, we will show that $\mathcal{X}_1 \subset \mathcal{X}_2$. To get a contradiction, suppose that the optima are $\mathcal{X}_1 = \{j\}$ and $\mathcal{X}_2 = \{k, l\}$, where we may assume that $t_k \leq t_l$. Optimality requires that
\fi
\begin{equation}v(\mathcal{X}_1 )  = f_j t_j > v(\{k\}) = f_k t_k\end{equation}
\ifen and \fi
\begin{align}
\begin{split}
v(\mathcal{X}_2) =  f_k (1- f_l) t_k + f_l t_l &> v(\{j, l\}) \\
& = f_j (1- f_l) t_j + (1- f_j) f_l t_l + f_j f_l \max\{t_j, t_l\} \\
&\geq  f_j (1- f_l) t_j + (1- f_j) f_l t_l + f_j f_l  t_l \\
&= f_j (1- f_l) t_j + f_l t_l  \\
&\geq f_k (1- f_l) t_k + f_l t_l  = v(\mathcal{X}_2)
\end{split}
\end{align}
\ifen which is a contradiction. 
\fi

\ifen (Inductive step.) Assume that $\mathcal{X}_1 \subset \cdots \subset \mathcal{X}_h$, and we will show $\mathcal{X}_h \subset \mathcal{X}_{h+1}$. Let $k = \argmax\{ t_k: k \in \mathcal{X}_{h+1}\}$ and write $\mathcal{X}_{h+1} = \mathcal{Y}_{h} \cup \{k\}$.
\fi

\ifen Suppose $k \notin \mathcal{X}_h$. To get a contradiction, suppose that $v(\mathcal{Y}_h) < v(\mathcal{X}_h)$ and  $v(\mathcal{X}_{h+1}) > v(\mathcal{X}_h \cup \{k\})$. Then
\fi
\begin{align}
\begin{split}
v(\mathcal{X}_{h+1})&= v(\mathcal{Y}_{h} \cup \{k\}) \\
&= (1 - f_k) v(\mathcal{Y}_h) + f_k t_k \\
&\leq (1 - f_k) v(\mathcal{X}_h) + f_k \operatorname{E}\bigl[ \max\{t_k, X_h\}\bigr]\\
&=  v(\mathcal{X}_h\cup \{k\})
\end{split}
\end{align}
\ifen is a contradiction.
\fi

\ifen Now suppose that $k \in \mathcal{X}_h$. We can write $\mathcal{X}_h = \mathcal{Y}_{h-1} \cup \{k\}$, where $ \mathcal{Y}_{h-1}$ is some portfolio of size $h-1$. It suffices to show that $ \mathcal{Y}_{h-1} \subset \mathcal{Y}_h$. By definition, $\mathcal{Y}_{h-1}$ (respectively, $\mathcal{Y}_{h}$) maximizes the function $v(\mathcal{Y}\cup\{k\})$ over portfolios of size $h-1$ (respectively, $h$) that do not include $k$. That is, $\mathcal{Y}_{h-1}$ and $\mathcal{Y}_h$ are the optimal complements to the singleton portfolio $\{k\}$.

Applying Lemma \ref{eliminationtheorem}, we eliminate $c_k$, transform the remaining $t_j$-values to $\bar t_j$ according to \eqref{howtotransformtj}, and obtain a function $w(\mathcal{Y}) = v(\mathcal{Y} \cup \{k\}) - f_k t_k$ that grades portfolios $\mathcal{Y} \subseteq \mathcal{C} \setminus \{k\}$ according to how well they complement $\{k\}$. Since $w(\mathcal{Y})$ is itself a portfolio valuation funtion and $\bar t_0 = 0$, the inductive hypothesis implies that $\mathcal{Y}_{h-1} \subset \mathcal{Y}_h$, which completes the proof.

\fi
\end{proof}

\ifen \subsection{Polynomial-time solution} \fi
\ifen Applying the result above yields an efficient greedy algorithm for the optimal portfolio: Start with the empty set and add schools one at a time, maximizing $v(\mathcal{X}\cup \{k\})$ at each addition. Sorting $t$ is  $O(m \log m)$.  At each of the $h$ iterations, there are $O(m)$ candidates for $k$, and computing $v(\mathcal{X}\cup \{k\})$ is $O(h)$ using \eqref{closedformportfoliovaluationX}; therefore, the time complexity of this algorithm is $O(h^2 m + m \log m)$. 
\fi
%

\ifen 
We reduce the computation time to $O(hm)$ by taking advantage of the transformation from Lemma \ref{eliminationtheorem}. Once school $k$ is added to $\mathcal{X}$, we eliminate it from the set $\mathcal{C}\setminus \mathcal{X}$ of candidates, and update the $t_j$-values of the remaining schools according to \eqref{howtotransformtj}. Now, the \emph{next} school added must be the optimal singleton portfolio in the modified market. But the optimal singleton portfolio consists simply of the school with the highest value of $f_j \bar t_j$. Therefore, by updating the $t_j$-values at each iteration according to \eqref{howtotransformtj}, we eliminate the need to compute $v(\mathcal{X})$ entirely. Moreover, this algorithm does not require the schools to be indexed in ascending order by $t_j$, which removes the $O(m\log m)$ sorting cost.
\fi
\ifen
Algorithm \ref{algorithmforlargeh} outputs a list $\mathtt{X}$ of the $h$ schools to which Alma should apply. The schools appear in the order of entry such that when the algorithm is run with $h=m$, the optimal portfolio of size $h$ is given by $\mathcal{X}_h = \{\mathtt{X[1]}, \dots, \mathtt{X[h]}\}$. The entries of the list $\mathtt{V}$ give the valuation thereof. 
\fi

\ifen {
\begin{algorithm}[h] 
\caption{Optimal portfolio algorithm for Alma’s problem.} \label{algorithmforlargeh}
\KwIn{Utility values $t \in(0, \infty)^m$, admissions probabilities $f \in (0, 1]^m$, application limit $h \leq m$.}
$\mathcal{C} \gets \{1 \dots m\}$\;
$\mathtt{X, V} \gets $ empty lists\;
\For{$i=1\dots h$}
{
    $k \gets \argmax_{j \in \mathcal{C}}\{f_j t_j\}$\;
    $\mathcal{C} \gets \mathcal{C} \setminus \{k\}$\;
    $\operatorname{append!}(\mathtt{X}, k)$\;
     \lIfElse{$i=1$}{$\operatorname{append!}(\mathtt{V}, f_k t_k)$}
     {$\operatorname{append!}(\mathtt{V}, \mathtt{V[i-1]} + f_k t_k)$}
    \For{$j \in \mathcal{C}$}
	{
	\lIfElse{$t_j \leq t_k$}{$t_j \gets (1 -  f_k) t_j $}{$t_j \gets  t_j -  f_k t_k$}
	}
}
\Return{$\mathtt{X, V}$}
\end{algorithm}
} \fi

\begin{theorem}[\ifen Validity of Algorithm \ref{algorithmforlargeh}\fi] \label{validityofalmaalgorithm}
\ifen Algorithm \ref{algorithmforlargeh} produces an optimal application portfolio for Alma's problem in $O(h m)$-time.
\fi
\end{theorem}
\begin{proof}
\ifen
Optimality follows from the proof of Theorem \ref{nestedapplication}. Suppose $\mathcal{C}$ is stored as a list. Then at each of the $h$ iterations of the main loop, finding the top school costs $O(m)$, and the $t_j$-values of the remaining $O(m)$ schools are each updated in unit time. Therefore, the overall time complexity is $O(h m)$.
\fi
\end{proof}

\ifen 
It is possible to store $\mathcal{C}$ as a binary max heap rather than a list. The heap is ordered according to the criterion $i \geq j \iff f_i t_i \geq f_j t_j$, and by draining the heap and reheapifying at the end of each iteration, the computation time remains $O(hm)$. However, in our numerical experiments, whose results are reported in Section \ref{numericalexperiments},
we found the list implementation to be much faster, because it is possible to identify the entering school $k$ as the utility parameters are updated, which all but eliminates the cost of the $\argmax\{\}$ operation. 
\fi

\ifen \subsection{Diminishing returns to application} \fi
\ifen The nestedness property implies that Alma's expected utility is a discretely concave function of $h$.
\fi

\begin{theorem}[\ifen Optimal portfolio valuation concave in $h$\fi] \label{concavityinh}
\ifen For $h = 2 \dots (m-1)$,
\fi
\begin{equation}v(\mathcal{X}_h) - v(\mathcal{X}_{h-1}) \geq v(\mathcal{X}_{h+1}) - v(\mathcal{X}_{h}).\end{equation} 
\end{theorem}

\begin{proof}
\ifen Applying Theorem \ref{nestedapplication}, we write $\mathcal{X}_h = \mathcal{X}_{h-1} \cup\{j\}$ and $\mathcal{X}_{h+1} = \mathcal{X}_{h-1} \cup\{j, k\}$. By optimality, $v(\mathcal{X}_h) - v(\mathcal{X}_{h-1}) \geq v(\mathcal{X}_{h-1}\cup\{k\}) - v(\mathcal{X}_{h-1})$. By submodularity and nestedness, $v(\mathcal{X}_{h-1}\cup\{k\}) - v(\mathcal{X}_{h-1}) \geq  v(\mathcal{X}_{h}\cup\{k\}) - v(\mathcal{X}_{h}) = v(\mathcal{X}_{h+1}) - v(\mathcal{X}_{h})$.
\fi
\end{proof}
\ifen
\noindent (An elementary proof is given in \S\,\ref{elementaryconcavityproof}.) It follows that when $\mathcal{X}_h$ is the optimal $h$-portfolio for a given market, $v(\mathcal{X}_h)$ is $O(h)$. In other words, Alma's problem exhibits diminishing returns to the application budget. Example \ref{tightexampleforoneoverhopt}, in which $v(\mathcal{X}_h)$ can be made arbitrarily close to $h$, establishes the tightness of this bound.
\fi


\ifen \subsection{A small example} \fi \label{planetsexamplesection}
\ifen 
Let us examine a fictional admissions market consisting of $m=8$ schools. The school data, along with the optimal solutions for each $h \leq m$, appear in Table \ref{planetsdata}.
\fi

\ifen
Below are shown the first several iterations of Algorithm \ref{algorithmforlargeh}. The values of $f_j$, $t_j$, and their product are recorded only for the schools remaining in $\mathcal{C}$. $f * t$, where $(f * t)_j = f_j t_j$, denotes the entrywise product of $f$ and $t$. The school added at each iteration, underlined, is the one whose $f_j t_j$-value is greatest.
\fi
\newcommand{\Iteration}{\ifen Iteration~\fi}
\begin{align*}
\text{\Iteration 1:}
&&C &= \{1, 2, 3, 4, 5, 6, 7, 8\} \\
&&f &= \{0.39, 0.33, 0.24, 0.24, 0.05, 0.03, 0.1, 0.12\} \\
&&t &= \{200, 250, 300, 350, 400, 450, 500, 550\} \\
&&f * t &= \{78.0, 82.5, 72.0, \underline{84.0}, 20.0, 13.5, 50.0, 66.0\} 
~~\tag*{\(\implies \mathcal{X}_3 = \{4\} \)}\\
\text{\Iteration 2:}
&&C &= \{1, 2, 3, 5, 6, 7, 8\} \\
&&f &= \{0.39, 0.33, 0.24, 0.05, 0.03, 0.1, 0.12\} \\
&&t &= \{152, 190, 228, 316, 366, 416, 466\} \\
&&f * t &= \{59.28, \underline{62.7}, 54.72, 15.8, 10.98, 41.6, 55.92\} 
~~\tag*{\(\implies \mathcal{X}_3 = \{4, 2\} \)}\\
\text{\Iteration 3:}
&&C &= \{1, 3, 5, 6, 7, 8\} \\
&&f &= \{0.39, 0.24, 0.05, 0.03, 0.1, 0.12\} \\
&&t &= \{101.84, 165.3, 253.3, 303.3, 353.3, 403.3\} \\
&&f * t &= \{39.718, 39.672, 12.665, 9.099, 35.33, \underline{48.396}\} 
~~\tag*{\(\implies \mathcal{X}_3 = \{4, 2, 8\} \)}\\
&&&\cdots\\
\text{\Iteration 8:}
&&C &= \{6\}, f = \{0.03\}, t = \{177.622\}, f * t= \{\underline{5.329}\}
~~\tag*{\(\implies \mathcal{X}_3 = \{4, 2, 8, 1, 7, 3, 5, 6\} \)}
\end{align*}
\ifen 
The output of the algorithm is $\mathtt{X} = [4, 2, 8, 1, 7, 3, 5, 6]$, and the optimal $h$-portfolio consists of its first $h$ entries. The ``priority'' column of Table \ref{planetsdata} shows the inverse permutation of $\mathtt{X}$, which is the minimum value of $h$ for which the school is included in the optimal portfolio. Figure \ref{planetsplot} shows the value of the optimal portfolio as a function of $h$. The concave shape of the plot suggests the result of Theorem \ref{concavityinh}. 
\fi

\begin{table}[h!] \centering
\small
\begin{tabular}{r|lcccc}
\ifen\textbf{Index $j$} & \textbf{School $c_j$} & \textbf{Admit prob. $f_j$} & \textbf{Utility $t_j$} & \textbf{Priority} & \textbf{$v(\mathcal{X}_h)$} \\ \hline
\fi
\\[-.75em]
1 & \ifen Mercury University   \fi   & 0.39   & 200 & 4   & 230.0   \\
2 & \ifen Venus University     \fi   & 0.33   & 250 & 2   & 146.7  \\
3 & \ifen Mars University      \fi   & 0.24   & 300 & 6   & 281.5  \\
4 & \ifen Jupiter University   \fi   & 0.24   & 350 & 1   & \phantom{0}84.0  \\
5 & \ifen Saturn University    \fi   & 0.05   & 400 & 7   & 288.8  \\
6 & \ifen Uranus University    \fi   & 0.03   & 450 & 8   & 294.1  \\
7 & \ifen Neptune University   \fi   & 0.10   & 500 & 5   & 257.7  \\
8 & \ifen Pluto College        \fi   & 0.12   & 550 & 3   & 195.1      
\end{tabular}
\caption{\label{planetsdata} \normalsize
\ifen College data and optimal application portfolios for a fictional market with $m=8$ schools. By the nestedness property (Theorem \ref{nestedapplication}), the optimal portfolio when the application limit is $h$ consists of the $h$ schools having priority $h$ or less.
\fi}
\end{table}

\begin{figure}[h!] 
 \centering
 \includegraphics[width=0.95\textwidth]{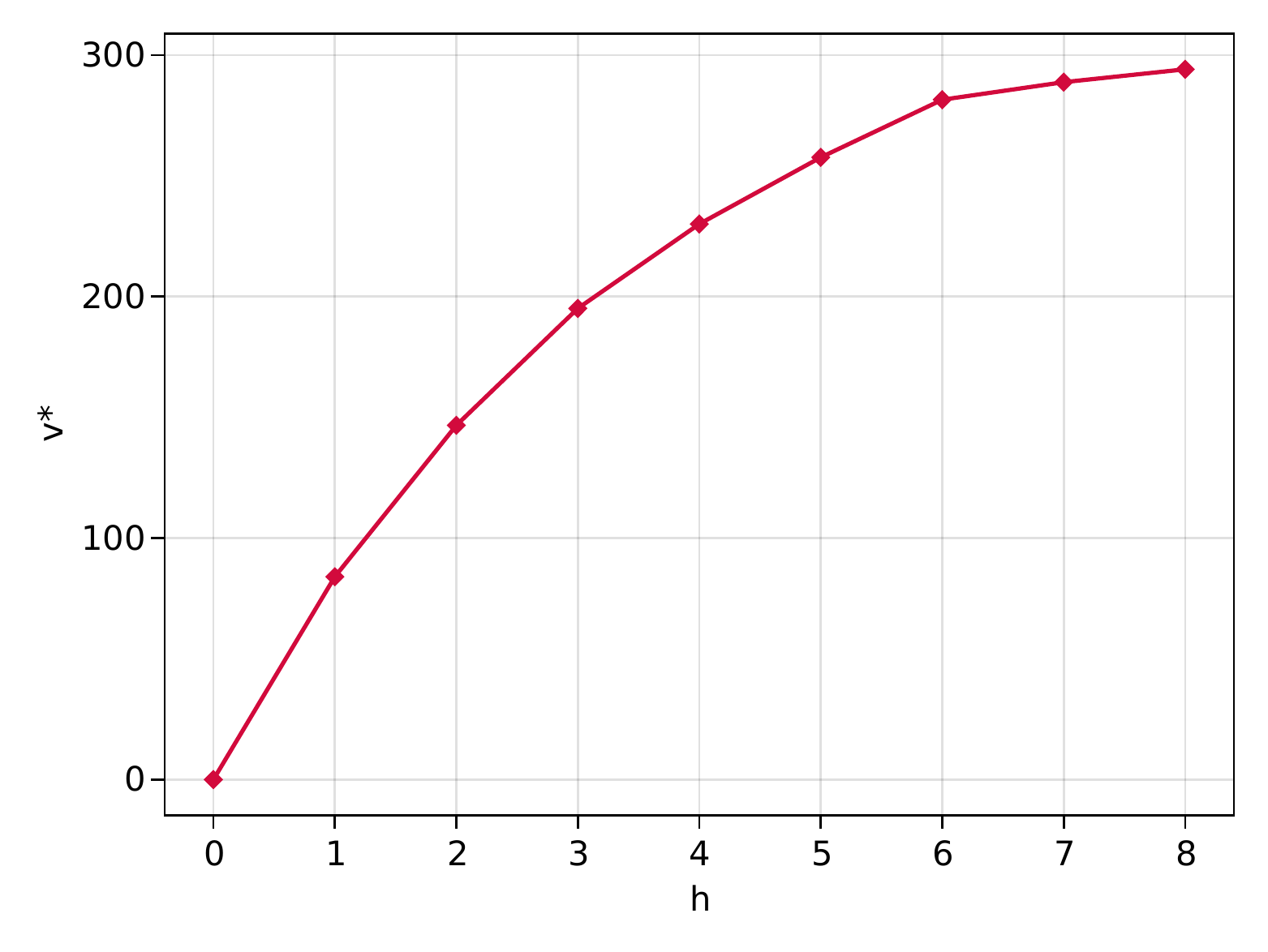}
  \caption{  \label{planetsplot}
  \ifen Application limit $h$ versus the optimal portfolio valuation $v^* = v(\mathcal{X}_h)$ in a fictional market with $m=8$ schools. The concave shape suggests the result of Theorem \ref{concavityinh}. 
  \fi}
\end{figure}

\ifen \section{Heterogeneous application costs} \fi\label{hetappcosts}
\ifen
Now we turn to the more general problem in which the constant $g_j$ represents the \emph{cost} of applying to $c_j$ and the student, whom we now call Ellis, has a \emph{budget} of $H$ to spend on college applications.  Applying Lemma \ref{assumetzerozero}, we assume $t_0 = 0$ throughout.
\fi

\begin{problem}[\ifen Ellis's problem\fi]
\ifen
Ellis's optimal college application portfolio is given by the solution to the following combinatorial optimization problem.
\fi
\begin{align}
\begin{split}
\text{maximize}\quad &  v(\mathcal{X}) = \sum_{j\in
\mathcal{X}} \Bigl(f_j t_j \prod_{\substack{i \in \mathcal{X}: \\ i > j}} (1 - f_{i}) \Bigr)\\
\text{subject to}\quad & \mathcal{X}\subseteq\mathcal{C},~~\sum_{j \in \mathcal{X}} g_j \leq H 
\end{split}
\end{align}
\end{problem}

\ifen
\noindent In this section, we show that this problem is NP-complete, then provide four algorithmic solutions: an exact branch-and-bound routine, an exact dynamic program, a fully polynomial-time approximation scheme (FPTAS), and a simulated-annealing heuristic. 
\fi

\ifen \subsection{NP-completeness} \fi
\ifen
The optima for Ellis's problem are not necessarily nested, nor is the number of schools in the optimal portfolio necessarily increasing in $H$. For example, if
$f = (0.5, 0.5, 0.5)$, $t = (1, 1, 219)$, and $g = (1, 1, 3)$,
then it is evident that the optimal portfolio for $H = 2$ is $\{1, 2\}$ while that for $H = 3$ is $\{3\}$. In fact, Ellis’s problem is NP-complete, as we will show by a transformation from the binary knapsack problem, which is known to be NP-complete (Garey and Johnson 1979, \S\,3.2.1).
\fi

\begin{problem}[\ifen Decision form of knapsack problem\fi]
\ifen 
An \emph{instance} consists of a set $\mathcal{B}$ of $m$ objects, utility values $u_j \in \mathbb{N}$ and weight $w_j \in \mathbb{N}$ for each $j \in \mathcal{B}$, and target utility $U\in \mathbb{N}$ and knapsack capacity $W\in \mathbb{N}$. The instance is called a \emph{yes-instance} if and only if there exists a set $\mathcal{B’} \subseteq \mathcal{B}$ having $\sum_{j \in \mathcal{B’}} u_j \geq U$ and  $\sum_{j \in \mathcal{B’}} w_j \leq W$.
\fi
\end{problem}

\begin{problem}[\ifen Decision form of Ellis's problem\fi] \label{ellisdecisionform}
\ifen
An \emph{instance} consists of an instance of Ellis’s problem and a target valuation $V$. The instance is called a \emph{yes-instance} if and only if there exists a portfolio $\mathcal{X} \subseteq \mathcal{C}$ having $v(\mathcal{X}) \geq V$ and  $\sum_{j \in \mathcal{X}} g_j \leq H$.
\fi
\end{problem}

\begin{theorem}
\ifen
The decision form of Ellis’s problem is NP-complete.
\fi
\end{theorem}

\begin{proof}
\ifen It is obvious that the problem is in NP.
\fi

\ifen
Consider an instance of the knapsack problem, and we will construct an instance of Problem \ref{ellisdecisionform} that is a yes-instance if and only if the corresponding knapsack instance is a yes-instance. Without loss of generality, we may assume that the objects in $\mathcal{B}$ are indexed in increasing order of $u_j$, that each $u_j >0$, and that each $w_j \leq W$. 
\fi

\ifen
Let $U_{\mathrm{max}} = \sum_{j \in \mathcal{B}} u_j$ and $\delta = {1} /{m U_{\mathrm{max}}} > 0$, and construct an instance of Ellis’s problem with $\mathcal{C} = \mathcal{B}$, $H = W$, all $f_j = \delta$, and each $t_j = u_j / \delta$. Clearly, $\mathcal{X} \subseteq \mathcal{C}$ is feasible for Ellis’s problem if and only if it is feasible for the knapsack instance. Now, we observe that for any nonempty $\mathcal{X}$,
\fi
\begin{align}
\begin{split}
\sum_{j \in \mathcal{X}} u_j &=  \sum_{j \in \mathcal{X}} f_j t_j
> \sum_{j \in \mathcal{X}} \Bigl( f_j t_j \prod_{\substack{j’ \in \mathcal{X}: \\ j' > j}} (1 - f_{j’}) \Bigr)
= v(\mathcal{X}) \\
&= \sum_{j \in \mathcal{X}} \Bigl( u_j  \prod_{\substack{j’ \in \mathcal{X}: \\ j' > j}} (1 - \delta) \Bigr)
\geq (1 - \delta)^m \sum_{j \in \mathcal{X}} u_j \\
&\geq (1 - m\delta) \sum_{j \in \mathcal{X}} u_j 
\geq \sum_{j \in \mathcal{X}} u_j  - m\delta U_{\mathrm{max}}
= \sum_{j \in \mathcal{X}} u_j  - 1.
\end{split}
\end{align}
\ifen
This means that the utility of an application portfolio $\mathcal{X}$ in the corresponding knapsack instance is the smallest integer greater than $v( \mathcal{X})$. That is, $\sum_{j \in \mathcal{X}} u_j \geq U$ if and only if $v(\mathcal{X}) \geq U -1$. Taking $V = U-1$ completes the transformation and concludes the proof. 
\fi
\end{proof}

\ifen 
An intuitive extension of the greedy algorithm for Alma's problem is to iteratively add to $\mathcal{X}$ the school $k$ for which $\bigl( v(\mathcal{X}\cup\{k\}) - v(\mathcal{X}) \bigr) / g_k$ is largest. But the construction above shows that the objective function of Ellis's problem can approximate that of a knapsack problem with arbitrary precision. Therefore, in pathological examples such as the following, the greedy algorithm can achieve an arbitrarily small approximation ratio. 
\fi
\begin{example} \label{greedyzeroellis}
\ifen
Let $t = (10, 2021)$, $f = (0.1, 0.1)$, $g = (1, 500)$, and $H = 500$. Then the greedy approximation algorithm produces the clearly suboptimal solution $\mathcal{X} = \{1\}$. 
\fi
\end{example} 
\ifen
We might expect to obtain a better approximate solution to Ellis's problem by solving the following knapsack problem (nevermind that it is NP-complete) as a surrogate.
\begin{equation}
\text{maximize}\quad \sum_{j \in \mathcal{X}} f_j t_j \qquad \text{subject to}\quad\mathcal{X}\subseteq\mathcal{C},~~\sum_{j \in \mathcal{X}} g_j \leq H
\end{equation}
However, upon inspection, this solution merely generalizes the na\"ive algorithm for Alma's problem (Definition \ref{naivealgorithm}). In Ellis's problem, because the number of schools in the optimal portfolio can be as large as $m-1$, the knapsack solution's approximation coefficient is $1/(m-1)$ by analogy with \eqref{oneoverhopt}. A tight example is as follows. 
\fi
\begin{example} \label{knapsackheuristicellis}
\ifen Consider the following instance of Ellis's problem, where $H = m-1$.
\fi
\begin{center}
\begin{tabular}{r|cccc}
$j$   & $1$            & $\cdots$ & $m-1$          & $m$            \\ \hline
$f_j$ & $1$            &  $\cdots$ & $1$            & $\frac{1}{m-1}$ \\
$t_j$ & $\frac{1}{m-1}$ &  $\cdots$ & $\frac{1}{m-1}$ & $m-1$          \\
$g_j$ & $1$            &  $\cdots$ & $1$            & $m-1$         
\end{tabular}%
\end{center}
\ifen
The feasible solutions $\mathcal{Y} = \{1, \dots, m-1\}$ and $\mathcal{X} = \{ m\}$ each have $\sum_{j \in \mathcal{Y}} f_j t_j = \sum_{j \in \mathcal{X}} f_j t_j =  1$, and thus the knapsack algorithm can choose $\mathcal{Y}$ with $v(\mathcal{Y}) = 1/(m-1)$. But the optimal solution is  $\mathcal{X}$ with $v(\mathcal{X}) =  1$. 
\fi
\end{example}

\ifen 
In summary, neither a greedy algorithm nor a knapsack-based algorithm solves Ellis's problem to optimality. 
\fi

\ifen \subsection{Branch-and-bound algorithm} \fi \label{subsectionbranchbound}
\ifen 
A traditional approach to integer optimization problems is the branch-and-bound framework, which generates subproblems in which the values of one or more decision variables are fixed and uses an upper bound on the objective function to exclude, or \emph{fathom,} branches of the decision tree that cannot yield a solution better than the best solution on hand (Martello and Toth 1990; Kellerer et al. 2004). In this subsection, we present an integer formulation of Ellis's problem and a linear program (LP) that bounds the objective value from above. We tighten the LP bound for specific subproblems by reusing the conditional transformation of the $t_j$-values from Algorithm \ref{algorithmforlargeh}. A branch-and-bound routine emerges naturally from these ingredients.
\fi

\ifen
To begin, let us characterize the portfolio $\mathcal{X}$ as the binary vector $x \in \{0, 1\}^m$, where $x_j = 1 $ if and only if $j \in \mathcal\{X\}$. Then it is not difficult to see that Ellis's problem is equivalent to the following integer nonlinear program. 
\fi
\begin{problem}[\ifen Integer NLP for Ellis's problem\fi] \label{integernlp}
\begin{align}
\begin{split}
\text{maximize}\quad &  v(x) = \sum_{j=1}^m \Bigl(f_j t_j  x_j \prod_{i > j} (1 - f_{i} x_i) \Bigr)\\
\text{subject to}\quad & \sum_{j=1}^m g_j x_j \leq H, ~~ x_i \in \{0, 1\}^m
\end{split}
\end{align}
\end{problem}
\ifen
\noindent Since the product in $v(x)$ does not exceed one, the following LP relaxation is an upper bound on the valuation of the optimal portfolio.
\fi
\begin{problem}[\ifen LP relaxation for Ellis's problem\fi] \label{LPrelaxation}
\begin{align}
\begin{split}
\text{maximize}\quad &  v_{\mathrm{LP}}(x) = \sum_{j=1}^m  f_j t_j x_j \\
\text{subject to}\quad & \sum_{j=1}^m g_j x_j \leq H, ~~ x \in [0, 1]^m
\end{split}
\end{align}
\end{problem}
\ifen
\noindent Problem \ref{LPrelaxation} is a continuous knapsack problem, which is easily solved in $O(m \log m)$-time by a greedy algorithm (Dantzig 1957). Balas and Zemel (1980, \S\,2) provide an $O(m)$ algorithm. 
\fi

\ifen In our branch-and-bound framework, a \emph{node} is characterized by a three-way partition of schools $\mathcal{C}= \mathcal{I} \cup \mathcal{O} \cup \mathcal{N}$ satisfying $\sum_{j \in \mathcal{I}} g_j \leq H$. $\mathcal{I}$ consists of schools that are ``in'' the application portfolio ($x_j = 1$), $\mathcal{O}$ consists of those that are ``out'' ($x_j = 0$), and $\mathcal{N}$ consists of those that are ``negotiable.'' 
The choice of partition induces a pair of subproblems. The first subproblem is an instance of Problem \ref{integernlp}, namely
\fi
\begin{align} \label{nodenlpsubproblem}
\begin{split}
\text{maximize}\quad &  v(x) = \gamma + \sum_{j \in \mathcal{N}} \Bigl(f_j \bar t_j x_j \prod_{\substack{i \in \mathcal{N}:\\i > j}} (1 - f_{i} x_i) \Bigr)\\
\text{subject to}\quad & \sum_{j \in \mathcal{N}} g_j x_j \leq \bar H; ~~ x_j \in \{0, 1\},~~j \in \mathcal{N}.
\end{split}
\end{align}
\ifen The second is the corresponding instance of Problem \ref{LPrelaxation}:
\fi
\begin{align} \label{nodelprelaxation}
\begin{split}
\text{maximize}\quad &  w_{\mathrm{LP}}(x)= \gamma + \sum_{j \in \mathcal{N}}  f_j \bar t_j  x_j \\
\text{subject to}\quad & \sum_{j \in \mathcal{N}} g_j x_j \leq \bar H; ~~ x_j \in [0, 1],~~j \in \mathcal{N}
\end{split}
\end{align}
\ifen
In both subproblems, $\bar H = H - \sum_{j\in \mathcal{I}} g_j$ denotes the residual budget. The parameters $\gamma$ and $\bar t$ are obtained by iteratively applying the transformation \eqref{howtotransformtj} to the schools in $\mathcal{I}$. For each $j\in \mathcal{I}$,  we increment $\gamma$ by the current value of $f_j \bar t_j$, eliminate $j$ from the market, and update the remaining $\bar t_j$-values using \eqref{howtotransformtj}. (An example is given below.)
\fi

\ifen 
Given a node $n  = (\mathcal{I}, \mathcal{O}, \mathcal{N})$, its children are generated as follows. Every node has two, one, or zero children. In the typical case, we select a school $j \in \mathcal{N}$ for which $g_j \leq \bar H$ and generate one child by moving $j$ to $\mathcal{I}$, and another child by moving $j$ it to $\mathcal{O}$. Equivalently, we set $x_j = 1$ in one child and $x_j = 0$ in the other. In principle, any school can be chosen for $j$, but as a greedy heuristic, we choose the school for which the ratio $f_j \bar t_j / g_j$ is highest. Notice that this method of generating children ensures that each node's $\mathcal{I}$-set differs from its parent's by at most a single school, so the constant $\gamma$ and transformed $\bar t_j$-values for the new node can be obtained by a single application of \eqref{howtotransformtj}. 
\fi

\ifen
There are two atypical cases. First, if every school in $\mathcal{N}$ has $g_j > \bar H$, then there is no school that can be added to $\mathcal{I}$ in a feasible portfolio, and the optimal portfolio on this branch is $\mathcal{I}$ itself. In this case, we generate only one child by moving all the schools from $\mathcal{N}$ to $\mathcal{O}$. Second, if $\mathcal{N} = \O$, then the node has zero children, and as no further branching is possible, the node is called a \emph{leaf.}
\fi

\begin{example}
\ifen
Consider a market in which $t = (20, 40, 60, 80, 100)$, $f = (0.5, 0.5, 0.5, 0.5, 0.5)$, $g = (3, 2, 3, 2, 3)$, and $H = 8$, and the node $n =  (\mathcal{I}, \mathcal{O}, \mathcal{N}) = (\{2, 5\}, \{1\}, \{3, 4\})$. Let us compute the two subproblems associated with $n$ and identify its children. To compute the subproblems, we first simply disregard $c_1$. Next, to eliminate $c_2$, we apply \eqref{howtotransformtj} to $t$ to obtain
$\{\bar{\bar t}_3, \bar{\bar t}_4, \bar{\bar t}_5\} = \{ (1 - f_2) t_3, (1 - f_2) t_4, (1 -f_2) t_5\} = \{30, 40, 50 \}$ and $\bar{\bar \gamma} = f_2 t_2 = 20$. We eliminate $c_5$ by again applying \eqref{howtotransformtj} to $\bar{ \bar t}$ to obtain $\{\bar t_3, \bar t_4\} =  \{\bar{\bar t}_3 - f_5 \bar{\bar t}_5, \bar{\bar t}_4-  f_5 \bar{\bar t}_5\}= \{5, 15\}$ and $\gamma = \bar{\bar \gamma} + f_5 \bar{\bar t}_5 = 35$. Finally, $\bar H = H - g_2 - g_5 = 3$. Now problems \eqref{nodenlpsubproblem} and \eqref{nodelprelaxation} are easily obtained by substitution.
\fi

\ifen
Since at least one of the schools in $\mathcal{N}$ has $g_j \leq \bar H$, $n$ has two children. Applying the node-generation rule, $c_4$ has the highest $f_j \bar t_j / g_j$-ratio, so the children are $n_1 = (\{2, 4, 5\}, \{1\}, \{3\})$ and $n_2 = (\{2, 5\}, \{1, 4\}, \{3\})$.
\fi
\end{example}

\ifen
To implement the branch-and-bound algorithm, we represent the set of candidate nodes---namely, nonleaf nodes whose children have not yet been generated---by $\mathfrak{T}$. Each time a node $n= (\mathcal{I}, \mathcal{O}, \mathcal{N})$ is generated, we record the values $v_{\mathcal{I}}[n] = v(\mathcal{I})$ and $v^*_{\mathrm{LP}}[n]$, the optimal objective value of the LP relaxation \eqref{nodelprelaxation}. Because $\mathcal{I}$ is a feasible portfolio, $v_{\mathcal{I}}[n]$ is a lower bound on the optimal objective value. Moreover, by the argument given in the proof of Theorem \ref{nestedapplication}, the objective function of  \eqref{nodenlpsubproblem} is identical to the function $v(\mathcal{I} \cup \mathcal{X})$. This means that $v_{\mathrm{LP}}^*[n]$  is an upper bound on the valuation of any portfolio that contains $\mathcal{I}$ as a subset and does not include any school in $\mathcal{O}$, and hence on the valuation of any portfolio on this branch. Therefore, if upon generating a new node $n_2$, we discover that its objective value $v_{\mathcal{I}}[n_2]$ is greater than $v_{\mathrm{LP}}^*[n_1]$ for some other node $n_1$, then we can disregard $n_1$ and all its descendants. 
\fi

\ifen
The algorithm is initialized by populating $\mathfrak{T}$ with the root node $n_0 = (\O, \O, \mathcal{C})$. At each iteration, it selects the node $n \in \mathfrak{T}$ having the highest $v_{\mathrm{LP}}^*[n]$-value, generates its children, and removes $n$ from the candidate set. Next, the children $n'$ of $n$ are inspected and added to the tree. If one of the children yields a new optimal solution, then we mark it as the best candidate and fathom any nodes $n''$ for which  $v_{\mathcal{I}}[n'] >  v_{\mathrm{LP}}^*[n'']$. When no nodes remain in the candidate set, the algorithm has explored every branch, so it returns the best candidate and terminates.
\fi

\ifen {
\begin{algorithm}[h] 
\caption{Branch and bound for Ellis's problem.} \label{ellisbnb}
\KwIn{Utility values $t \in(0, \infty)^m$, admissions probabilities $f \in (0, 1]^m$, application costs $g \in (0, \infty)^m$, budget $H \in (0, \infty)$.}
Root node $n_0 \gets (\O, \O, \mathcal{C})$\;
Current lower bound $L \gets 0$ and best solution $\mathcal{X} \gets \O$\;
Candidate set $\mathfrak{T} \gets \{ n_0\}$\;
\While{not finished}{
	\lIf{$\mathfrak{T} = \O$}{\Return{$\mathcal{X}, L$}}
	\Else{
		$n \gets \argmax\{ v_{\mathrm{LP}}^*[n] : n \in \mathfrak{T}\}$\;
		Remove $n$ from $\mathfrak{T}$\;
		\For{each child $n'$ of $n$}{
			\If{$L < v_{\mathcal{I}}[n']$}{
				$L \gets v_{\mathcal{I}}[n']$\;
				Update $\mathcal{X}$ to the $\mathcal{I}$-set associated with $n'$\;
			}
			\lIf{$v_{\mathrm{LP}}^*[n] > L$ and $n'$ is not a leaf}{add $n'$ to $\mathfrak{T}$}
		}
		\For{$n'' \in \mathfrak{T}$} {
			\lIf{$L >  v_{\mathrm{LP}}^*[n'']$}{remove $n''$ from $\mathfrak{T}$}
		}
	}
}
\end{algorithm}
} \fi

\begin{theorem}[\ifen Validity of Algorithm \ref{ellisbnb}\fi]
\ifen Algorithm \ref{ellisbnb} produces an optimal application portfolio for Ellis's problem.
\fi
\end{theorem}

\begin{proof}
\ifen
Because an optimal solution exists among the leaves of the tree, the discussion above implies that as long as the algorithm terminates, it returns an optimal solution.

To show that the algorithm does not cycle, it suffices to show that no node is generated twice. Suppose not: that two distinct nodes $n_1$ and $n_2$ share the same partition $(\mathcal{I}_{12}, \mathcal{O}_{12}, \mathcal{N}_{12})$. Trace each node's lineage up the tree and let $n$ denote the \emph{first} node at which the lineages meet. $n$ must have two children, or else its sole child is a common ancestor of $n_1$ and $n_2$, and one of these children, say $n_3$, must be an ancestor of $n_1$ while the other, say $n_4$, is an ancestor of $n_2$.  Write $n_3 = (\mathcal{I}_{3}, \mathcal{O}_{3}, \mathcal{N}_{3})$ and $n_4 = (\mathcal{I}_{4}, \mathcal{O}_{4}, \mathcal{N}_{4})$. By the node-generation rule, there is a school $j$ in $\mathcal{I}_3 \cap \mathcal{O}_4$, and the $\mathcal{I}$-set (respectively, $\mathcal{O}$-set) for any descendant of $\mathcal{I}_3$ (respectively, $\mathcal{O}_4$) is a superset of  $\mathcal{I}_3$ (respectively, $\mathcal{O}_4$). Therefore, $j \in \mathcal{I}_{12} \cap \mathcal{O}_{12}$, meaning that $(\mathcal{I}_{12}, \mathcal{O}_{12}, \mathcal{N}_{12})$ is not a partition of $\mathcal{C}$, a contradiction. 
\fi
\end{proof}

\ifen
The branch-and-bound algorithm is an interesting benchmark, but as a kind of enumeration algorithm, its computation time grows rapidly in the problem size, and unlike the approximation scheme we propose later on, there is no guaranteed bound on the approximation error after a fixed number of iterations. Moreover, when there are many schools in $\mathcal{N}$, the LP upper bound may be much higher than $v_{\mathcal{I}}[n']$. This means that significant fathoming operations do not occur until the algorithm has explored deep into the tree, at which point the bulk of the computational effort has already been exhausted. In our numerical experiments, Algorithm \ref{ellisbnb} was ineffective on instances larger than about $m=35$ schools, and therefore it does not represent a significant improvement over na\"ive enumeration. We speculate that the algorithm can be improved by incorporating cover inequalities (Wolsey 1998, \S\,9.3).  
\fi

\ifen \subsection{Pseudopolynomial-time dynamic program} \fi
\ifen
In this subsection, we assume, with a small loss of generality, that $g_j \in \mathbb{N}$ for $j = 1\dots m$ and $H \in\mathbb{N}$, and provide an algorithmic solution to Ellis's problem that runs in $O(Hm + m\log m)$-time. The algorithm resembles a familiar dynamic programming algorithm for the binary knapsack problem (Dantzig 1957; \emph{Wikipedia}, s.v. ``Knapsack problem''). Because we cannot assume that $H \leq m$ (as was the case in Alma's problem), this represents a pseudopolynomial-time solution (Garey and Johnson 1979, \S\,4.2). However, it is quite effective for typical college-application instances in which the application costs are small integers.
\fi

\ifen 
For $j = 0 \dots m$ and $h = 0 \dots H$, let $\mathcal{X}[j, h]$ denote the optimal portfolio using only the schools $\{ 1, \dots, j\}$ and costing no more than $h$, and let $V[j,h] = v(\mathcal{X}[j, h])$.  It is clear that if $j=0$ or $h=0$, then $\mathcal{X}[j, h] = \O$ and $V[j, h] = 0$.  For convenience, we also define $V[j, h] = -\infty$ for all $h < 0$.
\fi

\ifen
For the remaining indices, $\mathcal{X}[j, h]$ either contains $j$ or not. If it does not contain $j$, then $\mathcal{X}[j, h] = \mathcal{X}[j-1, h]$. On the other hand, if  $\mathcal{X}[j, h]$ contains $j$, then its valuation is $(1 - f_j) v(\mathcal{X}[j, h]\setminus \{j\}) + f_j t_j$. This requires that $\mathcal{X}[j, h]\setminus \{j\}$ make optimal use of the remaining budget over the remaining schools; that is, $\mathcal{X}[j, h] = \mathcal{X}[j-1, h - g_j] \cup\{j\}$. From these observations, we obtain the following Bellman equation for $ j = 1\dots m$ and $h = 1\dots H$:
\fi
\begin{align}
V[j, h] = \max\bigl\{ V[j-1, h], (1 - f_j) V[j-1, h-g_j] + f_j t_j \bigr\}
\end{align}
\ifen 
with the convention that $ -\infty \cdot 0 = -\infty$. The corresponding optimal portfolios are computed by observing that $\mathcal{X}[j, h]$ contains $j$ if and only if $V[j, h]> V[j-1, h]$. The optimal solution is given by $\mathcal{X}[m, H]$. The algorithm below performs these computations and outputs the optimal portfolio $\mathcal{X}$. 
\fi

\ifen {
\begin{algorithm}[h] 
\caption{Dynamic program for Ellis's problem with integral application costs.} \label{ellisDP1}
\KwIn{Utility values $t \in(0, \infty)^m$, admissions probabilities $f \in (0, 1]^m$, application costs $g \in \mathbb{N}^m$, budget $H \in\mathbb{N}$.}
Index schools in ascending order by $t$\;
Fill a lookup table with the values of $V[j, h]$\; \label{Vcreatedlookuptable}
$h \gets H$\;
$\mathcal{X} \gets \O$\;
\For{$j = m, m-1, \dots, 1$}{
	\If{$V[j-1, h] < V[j, h]$}{
		$\mathcal{X} \gets \mathcal{X}\cup\{j\}$\; 
		$h \gets h - g_j$\;
	}
}
\Return{$\mathcal{X}$}
\end{algorithm}
} \fi

\begin{theorem}[\ifen Validity of Algorithm \ref{ellisDP1}\fi]
\ifen 
Algorithm \ref{ellisDP1} produces an optimal application portfolio for Ellis's problem in $O(H m + m \log m)$-time.
\fi
\end{theorem}

\begin{proof}
\ifen
Optimality follows from the foregoing discussion. Sorting $t$ is $O(m \log m)$. The bottleneck step is the creation of the lookup table for $V[j, h]$ in line \ref{Vcreatedlookuptable}. Each entry is generated in unit time, and the size of the table is $O(Hm)$. 
\fi
\end{proof}

\ifen \subsection{Fully polynomial-time approximation scheme}\fi \label{fptashead}
\ifen 
As with the knapsack problem, Ellis's problem admits a complementary dynamic program that iterates on the value of the cheapest portfolio instead of on the cost of the most valuable portfolio. We will use this algorithm as the basis for an FPTAS for Ellis's problem that uses $O(m^3 / \varepsilon)$-time and -space.
\fi

\ifen 
We will represent approximate portfolio valuations using a fixed-point decimal with a precision of $P$, where $P$ is the number of digits to retain after the decimal point. Let $r[x] =  10^{-P}\lfloor 10^P x \rfloor$ denote the value of $x$ rounded down to its nearest fixed-point representation. Since $\bar U = \sum_{j\in \mathcal{C}} f_j t_j$ is an upper bound on the valuation of any portfolio, and since we will ensure that each fixed-point approximation is an underestimate of the portfolio's true valuation, the set $\mathcal{V}$ of possible valuations possible in the fixed-point framework is finite:
\fi
\begin{equation}
\mathcal{V} = \Bigl\{0, 1\times 10^{-P}, 2\times 10^{-P}, \dots, r\bigl[\bar U- 1\times 10^{-P}\bigr], r\bigl[\bar U\bigr]\Bigr\}
\end{equation}
\ifen Then $|\mathcal{V} | = \bar U \times 10^P + 1$.
\fi

\ifen
For the remainder of this subsection, unless otherwise specified, the word \emph{valuation} refers to a portfolio’s valuation within the fixed-point framework, with the understanding that this is an approximation. We will account for the approximation error below when we prove the dynamic program’s validity. 
\fi

\ifen
For integers $0 \leq j \leq m$ and $v \in [-\infty, 0) \cup \mathcal{V}$, let $\mathcal{W}[j, v]$ denote the least expensive portfolio that uses only schools $\{ 1, \dots, j\}$ and has valuation at least $v$, if such a portfolio exists. Denote its cost by $G[j, v] = \sum_{j\in \mathcal{W}[j, v]} g_j$, where $G[j, v] = \infty$ if $\mathcal{W}[j, v]$ does not exist. It is clear that if $v \leq 0$, then $\mathcal{W}[j, v] = \O$ and $G[j, h] = 0$, and that if $j = 0$ and $v > 0$, then $G[j, h] = \infty$.  For the remaining indices (where $j, v > 0$), we claim that
\fi
\begin{align} \label{recursionrelationforcostmindp}
G[j, v] &=
\begin{cases}
\infty, \quad & t_j < v \\
\min\bigl\{G[j-1, v], g_j + G[j-1, v - \Delta_j(v)] \bigr\}, \quad & t_j \geq v 
\end{cases}\\
\text{where}\qquad
\Delta_j (v) &= 
\begin{cases}
r\left[\frac{f_j}{1 - f_j} (t_j - v)\right], \quad & f_j < 1\\
\infty, &f_j = 1\ifen.\fi
\end{cases} \label{deltajvdef}
\end{align}
\ifen
In the $t_j < v$ case, any feasible portfolio must be composed of schools with utility less than $v$, and therefore its valuation can not equal $v$, meaning that $\mathcal{W}[j, v]$ is undefined. In the $t_j \geq v$ case, the first argument to $\min\{\}$ says simply that omitting $j$ and choosing $\mathcal{W}[j-1, v]$ is a permissible choice for $\mathcal{W}[j, v]$. If, on the other hand, $j \in \mathcal{W}[j, v]$, then
\fi
\begin{equation} \label{solvemeforvwjvminusj}
v(\mathcal{W}[j, v]) = (1 - f_j )v(\mathcal{W}[j, v]\setminus \{j\}) + f_j t_j\ifen.\fi\end{equation}
\ifen
Therefore, the subportfolio $\mathcal{W}[j, v]\setminus \{j\}$ must have a valuation of at least $v - \Delta$, where $\Delta$ satisfies $v = (1 - f_j )(v - \Delta) + f_j t_j $. When $f_j < 1$, the solution to this equation is $ \Delta = \frac{f_j}{1 - f_j} (t_j - v)$. By rounding this value down, we ensure that the true valuation of $\mathcal{W}[j, v]$ is \emph{at least} $v - \Delta$. When $t_j \geq v$ and $f_j = 1$, the singleton $\{j\}$ has $v(\{j\}) \geq v$, so
\fi
\begin{equation}G[j, v] = \min\bigl\{G[j-1, v], g_j \bigr\}.\end{equation}
\ifen
Defining $\Delta_j(v) = \infty$ in this case ensures that $g_j + G[j-1, v-\Delta_j(v)] = g_j+ G[j-1, v-\infty] = g_j $ as required.
\fi
%

\ifen
Once $G[j, v]$ has been calculated at each index, the associated portfolio can be found by applying the observation that $\mathcal{W}[j, v]$ contains $j$ if and only if $G[j, v] < G[j-1, v]$. Then an approximate solution to Ellis's problem is obtained by computing the largest achievable objective value $\max\{ w: G[m, w] \leq H\}$ and corresponding portfolio.
\fi

\ifen
\begin{algorithm}[h] 
\caption{Fully polynomial-time approximation scheme for Ellis's problem.} \label{ellisDP3}
\KwIn{Utility values $t \in  (0, \infty)^m$, admissions probabilities $f \in (0, 1]^m$, application costs $g \in (0, \infty)^m$, budget $H \in (0, \infty)^m$.}
\KwParams{Tolerance $\varepsilon \in (0, 1)$.}
Index schools in ascending order by $t$\;
Set precision $P \gets \bigl\lceil\log_{10}\left(m^2 / \varepsilon \bar U\right)\bigr\rceil$\;
Fill a lookup table with the entries of $G[j, h]$\; \label{createdlookuptable}
$v\gets  \max\{ w \in \mathcal{V} : G[m, w] \leq H\}$\; \label{vrecordedhere}
$\mathcal{X} \gets \O$\;
\For{$j = m, m-1, \dots, 1$}{
	\If{$G[j, v]< \infty$ and $G[j, v] < G[j-1, v]$}{
		$\mathcal{X} \gets \mathcal{X}\cup\{j\}$\; 
		$v \gets v -  \Delta_j(v)$\;
	}
}
\Return{$\mathcal{X}$}
\end{algorithm}
\fi

\begin{theorem}[\ifen Validity of Algorithm \ref{ellisDP3}\fi]
\ifen
Algorithm \ref{ellisDP3} produces a $(1 - \varepsilon)$-optimal application portfolio for Ellis's problem in $O(m^3 /\varepsilon)$-time. 
\fi
\end{theorem}

\begin{proof}
\ifen
(Optimality.) Let $\mathcal{W}$ denote the output of Algorithm \ref{ellisDP3} and $\mathcal{X}$ the true optimum. We know that $v(\mathcal{X}) \leq \bar U$, and because each singleton portfolio is feasible, $\mathcal{X}$ must be more valuable than the average singleton portfolio; that is, $v(\mathcal{X}) \geq \bar U / m$.
\fi

\ifen
Because $\Delta_j(v)$ is rounded down in the recursion relation defined by \eqref{recursionrelationforcostmindp} and \eqref{deltajvdef}, if $j \in \mathcal{W}[j, v]$, then the true value of $(1 - f_j) v\bigl(\mathcal{W}[j-1, v- \Delta_j(v)]\bigr) + f_j t_j$ may exceed the fixed-point valuation $v$ of $\mathcal{W}[j, v]$, but not by more than $10^{-P}$. This error accumulates additively with each school added to $\mathcal{W}$, but the number of additions is at most $m$. Therefore, where $v'(\mathcal{W})$ denotes the fixed-point valuation of $\mathcal{W}$ recorded in line \ref{vrecordedhere} of the algorithm, 
$v(\mathcal{W}) - v'(\mathcal{W}) \leq m 10^{-P}$.
\fi

\ifen
We can define $v'(\mathcal{X})$ analogously as the fixed-point valuation of $\mathcal{X}$ when its elements are added in index order and its valuation is updated and rounded down to the nearest multiple of $10^{-P}$ at each addition in accordance with \eqref{solvemeforvwjvminusj}. By the same logic, 
$v(\mathcal{X}) - v'(\mathcal{X}) \leq m 10^{-P}$. The optimality of $\mathcal{W}$ in the fixed-point environment implies that $v'(\mathcal{W}) \geq v'(\mathcal{X})$. 
\fi

\ifen
Applying these observations, we have
\fi
\begin{equation}
\begin{split}
v(\mathcal{W}) &\geq v'(\mathcal{W}) \geq v'(\mathcal{X})
\geq v(\mathcal{X}) - m 10^{-P}
\geq \left(1 - \frac{m^2 10^{-P}}{\bar U}\right) v(\mathcal{X})
\geq \left(1 - \varepsilon\right) v(\mathcal{X})
\end{split}
\end{equation}
\ifen
which establishes the approximation bound. 
\fi

\ifen
(Computation time.) The bottleneck step is the creation of the lookup table in line \ref{createdlookuptable}, whose size is $m \times |\mathcal{V}|$. Since
\fi
\begin{equation}
|\mathcal{V}| = \bar U \times 10^{P} + 1 = \bar U \times 10^ { \bigl\lceil\log_{10}\left(m ^2/ \varepsilon \bar U\right)\bigr\rceil} + 1
\leq\frac{m^2}{\varepsilon} \times \text{\ifen const.\fi}
\end{equation}
\ifen
is $O(m^2/ \varepsilon)$, the time complexity is as promised.
\fi
\end{proof}

\ifen
Since its time complexity is polynomial in $m$ and $1 / \varepsilon$, Algorithm \ref{ellisDP3} is an FPTAS for Ellis's problem (Vazirani 2001). 

Algorithms \ref{ellisDP1} and \ref{ellisDP3} can be written using recursive functions instead of lookup tables. However, since each function references itself \emph{twice,} the function values at each index must be recorded in a lookup table or memoized to prevent an exponential number of calls from forming on the stack.
\fi

\ifen \subsection{Simulated annealing heuristic} \fi\label{numericalexperiments}
\ifen
Simulated annealing is a popular local-search heuristic for nonconvex optimization problems. Because of their inherent randomness, simulated-annealing algorithms offer no accuracy guarantee; however, they are easy to implement, computationally inexpensive, and often produce solutions with a small optimality gap. In this section, we present a simulated-annealing algorithm for Ellis's problem. It is patterned on the procedure for integer programs outlined by Wolsey (1998, \S\,12.3.2).
\fi

\ifen
To generate an initial solution, our algorithm adopts the greedy heuristic of adding schools in descending order by $f_j t_j / g_j$ until the budget is exhausted. While Example \ref{greedyzeroellis} shows that the approximation coefficient of this solution can be arbitrarily small, in typical instances, it is much higher. Next, given a candidate solution $\mathcal{X}$, we generate a neighbor $\mathcal{X}_n$ by copying $\mathcal{X}$, adding schools to $\mathcal{X}_n$ until it becomes infeasible, then removing elements of $\mathcal{X}$ from $\mathcal{X}_n$ until feasibility is restored. If $v(\mathcal{X}_n) \geq v(\mathcal{X})$, then the candidate solution is updated to equal the neighbor; if not, then the candidate solution is updated with probability $\exp\bigl(v(\mathcal{X}_n) - v(\mathcal{X}) / T\bigr)$, where $T$ is a temperature parameter. This is repeated $N$ times, and the algorithm returns the best $\mathcal{X}$ observed.
\fi

\ifen
\begin{algorithm}[h] 
\caption{Simulated annealing for Ellis's problem.} \label{ellissimann}
\KwIn{Utility values $t  (0, \infty)^m$, admissions probabilities $f \in (0, 1]^m$, application costs $g \in (0, \infty)^m$, budget $H \in (0, \infty)^m$.}{}
\KwParams{Initial temperature $T \geq 0$, temperature reduction parameter $r \in (0, 1]$, number of iterations $N$.}
Add schools to $\mathcal{X}$ in descending order by $f_j t_j / g_j$ until the budget is exhausted\;
\For{$i = 1 \dots N$}{
	$\mathcal{X}_n \gets \operatorname{copy}(\mathcal{X})$\;
	Add random schools from $\mathcal{C} \setminus \mathcal{X}$ to $\mathcal{X}_n$ until  $\mathcal{X}_n$ infeasible\;
	Remove random schools in $\mathcal{X}$ from $\mathcal{X}_n$ until feasibility restored\;
	$\Delta = v(\mathcal{X}_n) - v(\mathcal{X})$\;
	\lIfElse{$\Delta \geq 0$}{$\mathcal{X} \gets \mathcal{X}_n$}
	{$\mathcal{X} \gets \mathcal{X}_n$ with probability $\exp(\Delta / T)$}
	$T \gets rT$\;
}
\Return{the best $\mathcal{X}$ found}
\end{algorithm}
\fi

\begin{proposition}
\ifen The computation time of Algorithm \ref{ellissimann} is $O(Nm)$. 
\fi
\end{proposition}

\ifen
\noindent In our numerical experiments, the output of Algorithm \ref{ellissimann} was typically quite close to the optimum.
\fi


\ifen \section{Numerical experiments} \fi\label{numericalexperiments}
\ifen
In this section, we present the results of numerical experiments designed to test the practical efficacy of the algorithms derived above. 
\fi

\ifen \subsection{Implementation notes} \fi
\ifen
We chose to implement our algorithms in the Julia language (v1.8.0b1) because its system of parametric data types allows the compiler to optimize for differential cases such as when the $t_j$-values are integers or floating-point numbers. Julia also offers convenient macros for parallel computing, which enabled us to solve more and larger problems in the benchmark (Bezanson et al. 2017). We made extensive use of the DataStructures.jl package (v0.18.11, \url{github.com/JuliaCollections}). The code is available at \url{github.com/maxkapur/OptimalApplication}.
\fi

\ifen
The dynamic programs, namely Algorithms \ref{ellisDP1} and \ref{ellisDP3}, were implemented using recursive functions and dictionary memoization, as described in Cormen et al. (1990, \S\,16.6). Our implementation of Algorithm \ref{ellisDP3} also differed from that described in Subsection \ref{fptashead} in that we represented portfolio valuations in \emph{binary} rather than decimal, with the definitions of $P$ and $\mathcal{V}$ modified accordingly, and instead of fixed-point numbers, we worked in integers by multiplying each element of $\mathcal{V}$ by $2^P$. These modifications yield a substantial performance improvement without changing the fundamental algorithm design or complexity analysis.
\fi

\ifen \subsection{Experimental procedure} \fi
\ifen
We conducted three numerical experiments. Experiments 1 and 2 evaluate the computation times of our algorithms for Alma's problem and Ellis's problem, respectively. Experiment 3 investigates the empirical accuracy of the simulated-annealing heuristic as compared to an exact algorithm. 
\fi

\def\nmarkets{50}
\ifen
To generate synthetic markets, we drew the $t_j$-values independently from an exponential distribution with a scale parameter of ten and rounded up to the nearest integer. To achieve partial negative correlation between $t_j$ and $f_j$, we then set $f_j = 1 / (t_j + 10\,Q)$, where $Q$ is drawn uniformly from the interval $[0, 1)$. In Experiment 1, which concerns Alma's problem, we set $h = \lfloor m/ 2 \rfloor$. In Experiments 2 and 3, which concern Ellis's problem, each $g_j$ is drawn uniformly from the set $\{5, \dots, 10\}$ and we set $H = \lfloor \frac{1}{2} \sum g_j \rfloor$. This cost distribution resembles real-world college application fees, which are often multiples of \$10 between \$50 and \$100. A typical instance is shown in Figure \ref{samplemarket}.
\fi

\ifen
The experimental variables in Experiments 1 and 2 were the market size $m$, the choice of algorithm, and (for Algorithm \ref{ellisDP3}) the tolerance $\varepsilon$. For each combination of the experimental variables, we generated \nmarkets~markets, and the optimal portfolio was computed three times, with the fastest of the three repetitions recorded as the computation time. We then report the mean and standard deviation across the \nmarkets~markets. Therefore, each cell of each table below represents a statistic over {\the\numexpr 3 * \nmarkets \relax}~computations. Where applicable, we do not count the time required to sort the entries of $t$. As the simulated-annealing heuristic's computation time is negligible, it was excluded from Experiment 2. 
\fi

\ifen
In Experiment 3, we generated 500 markets with sizes varying uniformly in the logarithm between $m = 2^3$ and $2^{11}$. For each market, we computed a heuristic solution using Algorithm \ref{ellissimann} and an exact solution using Algorithm \ref{ellisDP1} and calculated the optimality ratio. We chose to use $N = 500$ iterations in the simulated-annealing algorithm, and the temperature parameters $T = 1/4$ and $r = 1/16$ were selected by a preliminary grid search.  
\fi

\ifen \subsection{Summary of results} \fi
\ifen
Experiment 1 compared the performance of Algorithm \ref{algorithmforlargeh} for homogeneous-cost markets of various sizes when the set of candidate schools $\mathcal{C}$ is stored as a list and as a binary max heap ordered by the $f_j \bar t_j$-values. The results appear in Table \ref{experiment1results}. Our results indicate that the list implementation is faster. 
\fi


\ifen
In Experiment 2, we turned to the general problem, and compared the performance of the exact algorithms (Algorithms \ref{ellisbnb} and \ref{ellisDP1}) and the approximation scheme (Algorithm \ref{ellisDP3}) at tolerances 0.5 and 0.05.  The results, which appear in Table \ref{experiment2results}, broadly agree with the time complexity analyses presented above. The branch-and-bound algorithm proved impractical for even medium-sized instances. Overall, we found the exact dynamic program to be the fastest algorithm, while the FPTAS was rather slow, a result that echoes the results of computational studies on knapsack problems (Martello and Toth 1990, \S\,2.10). The strong performance of Algorithm \ref{ellisDP1} is partly attributable to the structure of our synthetic instances, in which application costs are small integers and $H$ is proportional in expectation to $m$, meaning the expected computation time is $O(m^2)$. However, the relative advantage of Algorithm \ref{ellisDP1} may be even more pronounced in real college-application instances because the typical student's application budget accommodates at most a dozen or so schools and is constant in $m$. 
\fi

\ifen
The results of Experiment 3, which evaluated the performance of the simulated-annealing heuristic in our synthetic instances, are plotted in Figure \ref{experiment3results}. In the every instance, the heuristic found a solution within 10 percent of optimality, and in the vast majority of instances, it was within 2 percent of optimality. The heuristic performs most poorly in small markets of size $m \leq 16$, but as exact algorithms are tractable at this scale, this result presents no cause for concern.
\fi

\newcommand{\lastptofcaption}{\ifen
For each value of $m$, \nmarkets~markets were generated, and the computation time was recorded as fastest of three repetitions of the algorithm. The table shows the average time (standard deviation) over the \nmarkets~instances.
\fi}

\begin{figure}[h!] 
\centering
\includegraphics[width=0.85\textwidth]{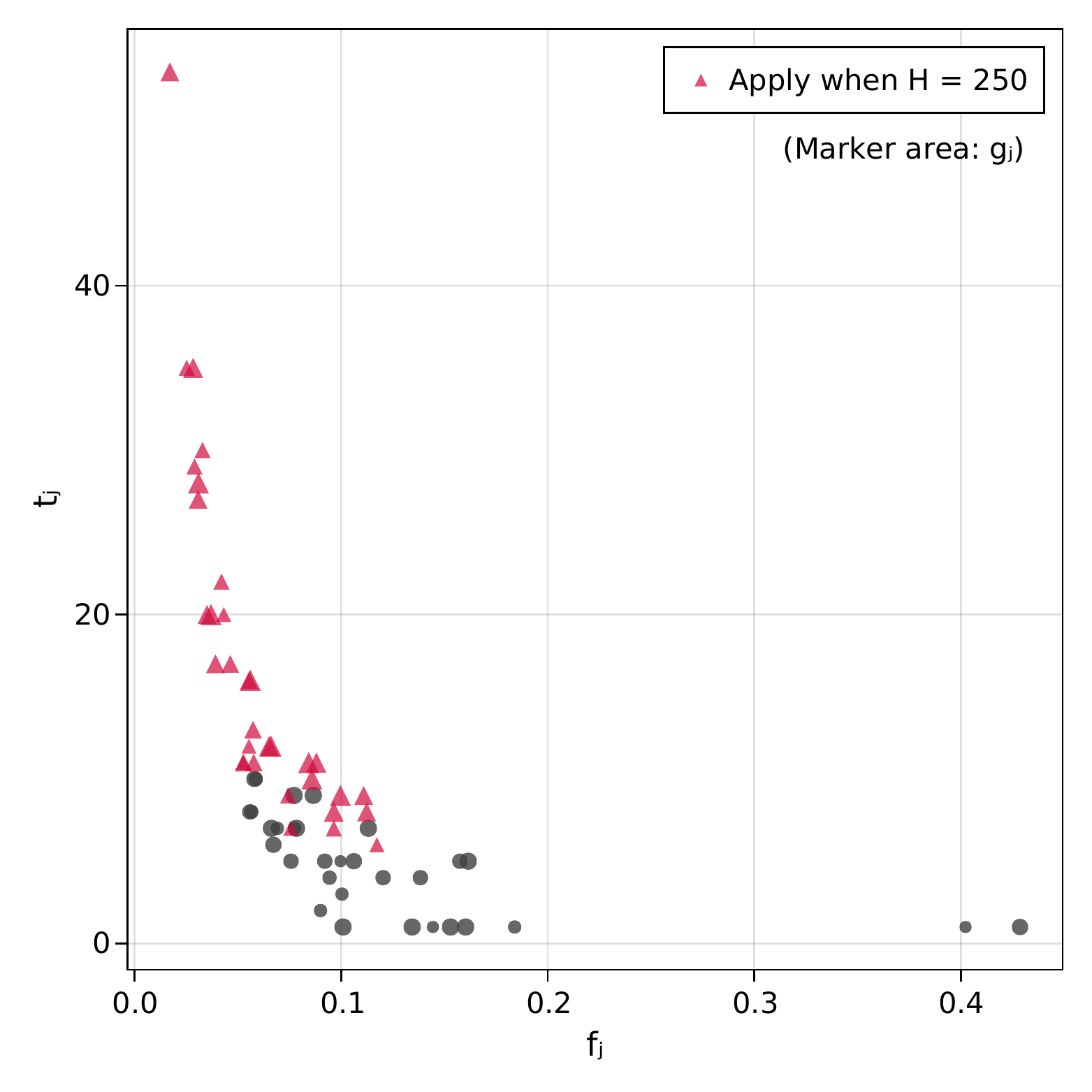}
  \caption{\label{samplemarket}
 \ifen A typical randomly-generated instance with $m=64$ schools and its optimal application portfolio. The application costs $g_j$ were drawn uniformly from $\{5, \dots, 10\}$. The optimal portfolio was computed using Algorithm \ref{ellisDP1}.
 \fi}
\end{figure}

\begin{table}[h!] \centering
\small
\begin{tabular}{r|r@{~}r|r@{~}r}
\ifen 
 \textbf{\begin{tabular}[r]{@{}r@{}}Number of\\schools $m$\end{tabular}}& \multicolumn{2}{c|}{\textbf{\begin{tabular}[c]{@{}c@{}}Algorithm \ref{algorithmforlargeh}\\with list\end{tabular}}}  & \multicolumn{2}{c}{\textbf{\begin{tabular}[c]{@{}c@{}}Algorithm \ref{algorithmforlargeh}\\with heap\end{tabular}}}\\ \hline
 \fi
    16 &    0.00 &  (0.00) &    0.01 &    (0.00) \\
    64 &    0.01 &  (0.00) &    0.08 &    (0.02) \\
   256 &    0.06 &  (0.00) &    0.96 &    (0.26) \\
  1024 &    0.82 &  (0.03) &   14.27 &    (2.28) \\
  4096 &   12.87 &  (0.71) &  230.77 &   (18.75) \\
 16384 &  199.48 &  (1.77) & 3999.19 &  (283.48)
\end{tabular}
\caption{\label{experiment1results} \normalsize
\ifen (Experiment 1.) Time in ms to compute an optimal portfolio for an admissions market with homogeneous application costs using Algorithm \ref{algorithmforlargeh} when $\mathcal{C}$ is stored as a list and as a heap. \lastptofcaption~In every case, $h = m/2$. 
\fi}
\end{table}

\begin{table}[h!] \centering
\small
\begin{tabular}{r|r@{~}r|r@{~}r|r@{~}r|r@{~}r}
\ifen
\textbf{\begin{tabular}[r]{@{}r@{}}Number of\\schools $m$\end{tabular}}&\multicolumn{2}{c|}{\textbf{\begin{tabular}[c]{@{}c@{}}Algorithm \ref{ellisbnb}:\\Branch \& bound\end{tabular}}}  & \multicolumn{2}{c|}{\textbf{\begin{tabular}[c]{@{}c@{}}Algorithm \ref{ellisDP1}:\\Costs DP\end{tabular}}}  &\multicolumn{2}{c|}{\textbf{\begin{tabular}[c]{@{}c@{}}Algorithm \ref{ellisDP3}:\\FPTAS, $\varepsilon= 0.5$\end{tabular}}}  & \multicolumn{2}{c}{\textbf{\begin{tabular}[c]{@{}c@{}}Algorithm \ref{ellisDP3}:\\FPTAS, $\varepsilon= 0.05$\end{tabular}}}   \\ \hline
\fi
   8 &  0.02 & (0.01) &  0.01 & (0.00) &   0.05 &   (0.01) &     0.16 &    (0.04) \\
  16 &  0.10 & (0.04) &  0.07 & (0.02) &   0.39 &   (0.10) &     2.59 &    (0.62) \\
  32 & 12.43 & (9.88) &  0.29 & (0.05) &   2.15 &   (0.29) &    33.07 &   (10.72) \\
  64 &     — &    (—) &  1.24 & (0.16) &  13.24 &   (2.22) &   339.64 &  (100.38) \\
 128 &     — &    (—) &  6.20 & (0.64) &  79.92 &  (20.18) &  2042.63 &  (749.71) \\
 256 &     — &    (—) & 30.63 & (2.28) & 818.50 & (632.98) & 18949.50 & (3533.88)
\end{tabular}%
\caption{\label{experiment2results} \normalsize
\ifen (Experiment 2.) Time in ms to compute an optimal or $(1 - \varepsilon)$-optimal portfolio for an admissions market with heterogeneous application costs using the three algorithms developed in Section \ref{hetappcosts}.
The branch-and-bound algorithm is impractical for large markets. \lastptofcaption
\fi}
\end{table}

\begin{figure}[h!] 
\centering
\includegraphics[width=0.95\textwidth]{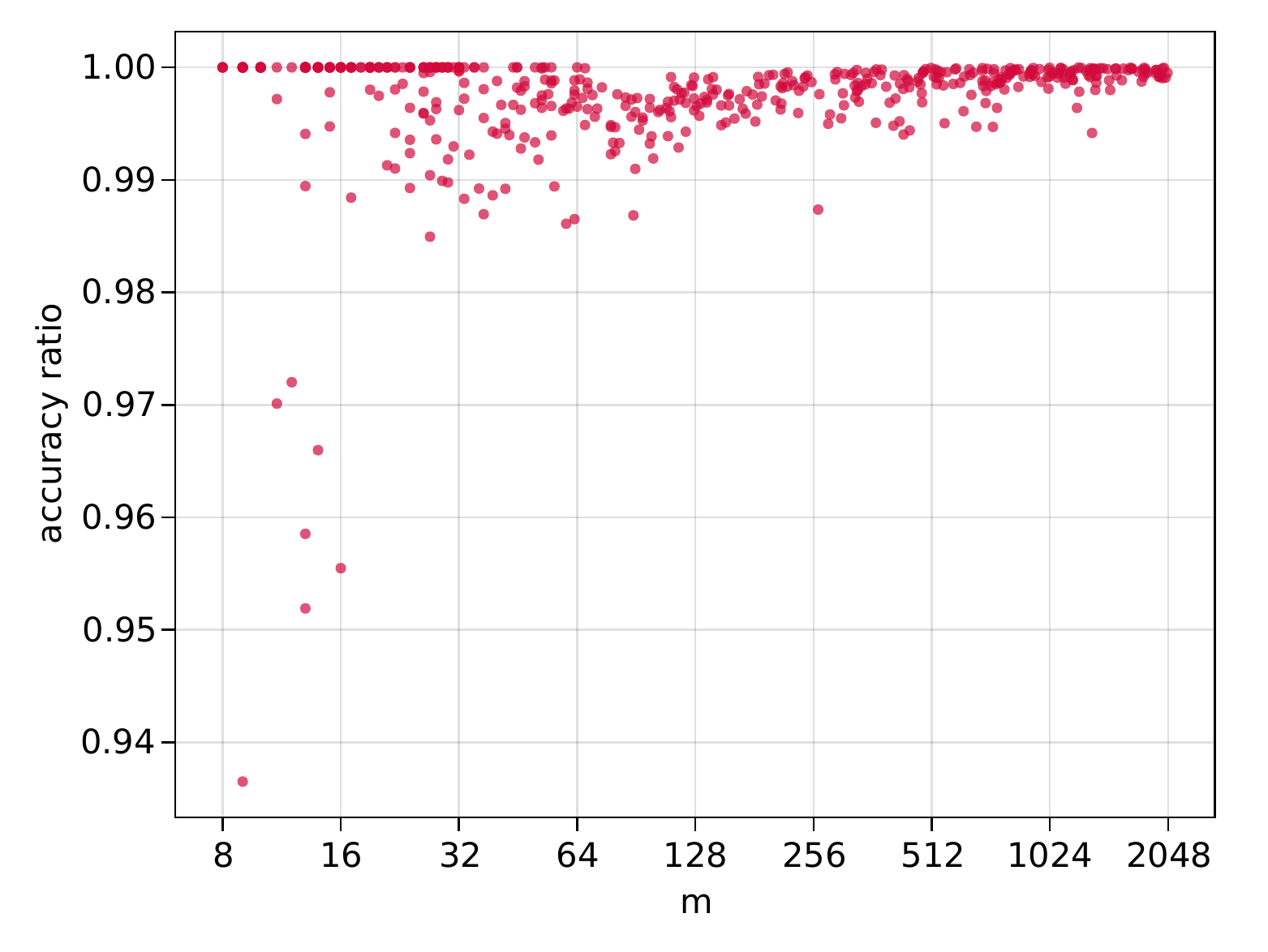}
  \caption{\label{experiment3results}
 \ifen (Experiment 3.) Optimality ratio achieved by simulated annealing (Algorithm \ref{ellissimann}) in markets of varying size, with parameters $N = 500$, $T = 1/4$, and $r = 1/16$. Optimal portfolios were computed using Algorithm \ref{ellisDP1}. 
 \fi}
\end{figure}

\ifen \section{Conclusion and ideas for future research} \fi\label{conclusion}
\ifen
This study has introduced a novel combinatorial optimization problem that we call the college application problem. It can be viewed as a kind of asset allocation or knapsack problem with a submodular objective. We showed that the special case in which colleges have identical application costs can be solved in polynomial time by a greedy algorithm because the optimal solutions are nested in the budget constraint. The general problem is NP-complete. We provided four solution algorithms. The strongest from a theoretical standpoint is an FPTAS that produces a $(1-\varepsilon)$-approximate solution in $O(m^3 / \varepsilon)$-time. On the other hand, for typical college-application instances, the dynamic program based on application expenditures is both easier to implement and substantially more time efficient. A heuristic simulated-annealing algorithm also exhibited strong empirical performance in our computational study. The algorithms discussed in this paper are summarized in Table \ref{algorithmssummary}.
\fi

\ifen
\begin{table}[h!] \centering
\small
\begin{tabular}{r|lllll}
\textbf{Algorithm} & \textbf{Reference}  & \textbf{Problem} & \textbf{Restrictions} & \textbf{Exactness}       & \textbf{Computation time} \\ \hline
\xrowht[()]{1.5em}  \begin{tabular}[r]{@{}r@{}}Na\"ive\end{tabular} & Definition \ref{naivealgorithm}                   & \begin{tabular}[l]{@{}l@{}}Homogeneous \\ costs\end{tabular}     & None                  & $(1/h)$-opt.               & $O(m)$                    \\ 
\xrowht[()]{1.5em}  Greedy                & Algorithm \ref{algorithmforlargeh} &  \begin{tabular}[l]{@{}l@{}}Homogeneous \\ costs\end{tabular}    & None                  & Exact                    & $O(hm)$                   \\
\xrowht[()]{1.5em}  \begin{tabular}[r]{@{}r@{}}Branch and\\ bound\end{tabular}   & Algorithm \ref{ellisbnb} & General          & None                  & Exact                    & $O(2^m)$                  \\
\xrowht[()]{1.5em}  Costs DP     &  Algorithm \ref{ellisDP1}& General          & $g_j$ integer         & Exact                    & $O(Hm + m \log m)$        \\
\xrowht[()]{1.5em}  FPTAS      &    Algorithm \ref{ellisDP3} & General          &  None         & $(1 - \varepsilon)$-opt. & $O(m^3 / \varepsilon)$   \\
\xrowht[()]{1.5em}  \begin{tabular}[r]{@{}r@{}}Simulated\\annealing\end{tabular}        &    Algorithm \ref{ellissimann} & General          &  None         & $0$-opt. & $O(Nm)$   
\end{tabular}%
\caption{\label{algorithmssummary} \normalsize
Summary of algorithms discussed in this paper.}
\end{table}
\fi

\ifen 
Three extensions of this problem appear amenable to future study.
\fi

\ifen \subsection{Explicit treatment of risk aversion}\fi
\ifen
The familiar Markowitz portfolio optimization model includes an explicit risk-aversion term, whereas our model has settled for an implicit treatment of risk, namely the tradeoff between selective schools and safety schools inherent in the maximax objective function. However, it is possible to augment the objective function $v(\mathcal{X}) = \operatorname{E}[X]$ to incorporate a variance penalty $\beta \geq 0$ as follows:
\fi
\begin{align}
\begin{split}
v_{\mathrm{\beta}}(\mathcal{X}) &=  \operatorname{E}[X] - \beta \operatorname{Var}(X) \\
&=   \operatorname{E}[X] - \beta \left(  \operatorname{E}[X^2]  -  \operatorname{E}[X]^2 \right) \\
& = \sum_{j\in\{0\}\cup\mathcal{X}} \Bigl( f_j t_j \prod_{\substack{i \in \mathcal{￣X}: \\ i > j}} (1 - f_{i}) \Bigr)
 - \beta \sum_{j\in\{0\}\cup\mathcal{X}} \Bigl( f_j t_j^2 \prod_{\substack{i \in \mathcal{X}: \\ i > j}} (1 - f_{i}) \Bigr)
  + \beta v(\mathcal{X})^2 \\
  &= v(\mathcal{X}; \tau) +  \beta v(\mathcal{X}; t)^2
\end{split}
\end{align}
\ifen
where $\tau_j = t_j - \beta t_j^2$. Since the first term is itself a portfolio valuation function, and the second is a monotonic transformation of one, we speculate that one of the algorithms given above could be used as a subroutine to trace out the efficient frontier by maximizing $v(\mathcal{X}; t)$ subject to the budget constraint and $v(\mathcal{X}; \tau) \geq \alpha$ for various values of $\alpha \geq 0$. 
\fi

\ifen
A parametric treatment of risk aversion may align our model more closely to real-world applicant behavior. Conventional wisdom asserts that the best college application strategy combines reach, target, and safety schools in roughly equal proportion (Jeon 2015). When the application budget is small relative to the size of the admissions market, our algorithms recommend a similar approach (see the example of Subsection \ref{planetsexamplesection}). However, inspecting equation \eqref{howtotransformtj}, which discounts school utility parameters to reflect their marginal value relative to the schools already in the portfolio, reveals that low-utility schools are penalized more harshly than high utility schools in the marginal analysis. Consequentially, as the application budget grows, the optimal portfolio in our model tends to favor reach schools over safety schools. (See Figure \ref{samplemarket} for a clear illustration of this phenomenon.) 
 
A potential application of our model is an econometric study that estimates students’ perceptions of college quality ($t_j$) from their observed application behavior ($\mathcal{X}$) under the assumption of utility maximization. Empirical research indicates that heterogeneous risk aversion is a significant source of variation in human decision-making, especially in the domains of educational choice and employment search (Kahneman 2011; Hartlaub and Schneider 2012; van Huizen and Alessie 2019). Therefore, an endogenous risk-aversion parameter would greatly enhance our model’s explanatory power in the econometric setting.
\fi

\ifen \subsection{Signaling strategies}\fi
\ifen
Another direction of further research with immediate application in college admissions is to incorporate additional signaling strategies into the problem. In the Korean admissions process, the online application form contains three multiple-choice fields, labeled \emph{ga, na,} and \emph{da} for the first three letters of the Korean alphabet, which students use to indicate the colleges to which they wish to apply. Most schools appear only in one or two of the three fields. Therefore, students are restricted not only in the number of applications they can submit, but in the combinations of schools that may coincide in a feasible application portfolio. From a portfolio optimization perspective, this is a \emph{diversification constraint,} because its primary purpose is to prevent students from applying only to top-tier schools. However, in addition to overall selectivity, the three fields also indicate different evaluation schemes: The \emph{ga} and \emph{na} customarily indicate the applicant's first and second choice, and applications filed under these fields are evaluated with greater emphasis on standardized-test scores than those filed under the \emph{da} field. Some colleges appear in multiple fields and run two parallel admissions processes, each with different evaluation criteria. Therefore, if a student recognizes that she has a comparative advantage in (for example) her interview skills, she may increase her chances of admission to a competitive school by applying under the \emph{da} field. The optimal application strategy in this context can be quite subtle, and it is a subject of perennial debate on Korean social networks.

An analogous feature of the American college admissions process is known as early decision, in which at the moment of application, a student commits to enrolling in a college if admitted. In principle, students who apply with early decision may obtain a better chance of admission by signaling their eagerness to attend, but doing so weakens the college's incentive to entice the student with discretionary financial aid, altering the utility calculus.
\fi

\ifen
In the integer formulation of the college application problem (Problem \ref{integernlp}), a natural way to model these signaling strategies without introducing too much complexity is to split each $c_j$ into $c_{j+}$ and $c_{j-}$. The binary decision variables are $x_{j+} = 1$ if the student applies to $c_j$ with high priority (that is, in  the \emph{ga} or \emph{na} field or with early decision) and $x_{j-} = 1$ if applying with low priority (that is, in the \emph{da} field or without early decision). Now, assuming that the differential admission probabilities $f_{j+}$ and $f_{j-}$ and utilities $t_{j+}$ and $t_{j-}$ are known, adding the logical constraints 
\begin{align}
 x_{j+} + x_{j-} \leq 1, ~~ j = 1\dots m\qquad\text{and}\qquad\sum_{j=1}^m x_{j+} \leq 1
\end{align}
completes the formulation. These are knapsack constraints; therefore, the submodular maximization algorithm of Kulik et al. (2013) can be used to obtain a $(1 - 1/e - \varepsilon)$-approximate solution for this problem. When the budget constraint is a cardinality constraint (as in Alma's problem), we note that adding these constraints yields a matroidal solution set, and Calinescu et al. (2011)'s algorithm yields an alternative solution, which has the same approximation coefficient of $1 - 1/e -\varepsilon$. 
\fi

\ifen \subsection{Memory-efficient dynamic programs}\fi
\ifen
Our numerical experiments suggest that the performance of the dynamic programming algorithms is bottlenecked by not computation time, but memory usage. Reducing these algorithms' storage requirements would enable us to solve considerably larger problems.

Abstractly speaking, Algorithms \ref{ellisDP1} and \ref{ellisDP3} are two-dimensional dynamic programs that represent the optimal solution by $Z[N, C]$. Here $N$ is the number of decision variables, $C$ is a constraint parameter, and $Z[n, c]$ is the optimal objective value achievable using the first $n\leq N$ decision variables when the constraint parameter is $c \leq C$. The algorithm iterates on a recursion relation that expresses $Z[n, c]$ as a function of $Z[n -1, c]$ and $Z[n -1, c']$ for some $c' \leq c$. (In Algorithm \ref{ellisDP1}, $Z$ is the maximal portfolio valuation, $N = m$, and $C =H$; in Algorithm \ref{ellisDP3},  $Z$ is the minimal application expenditures, $N = m$, and $C = |\mathcal{V}| \propto m^2 / \varepsilon$.)

When such a dynamic program is implemented using a lookup table or dictionary memoization, producing the optimal solution requires $O(NC)$-time and -space. Kellerer et al. (2004, \S\,3.3) provide a technique for transforming the dynamic program $Z$ into a divide-and-conquer algorithm that produces the optimal solution in $O(N C)$-time and $O(N + C)$-space, a significant improvement. However, their technique requires the objective function to be additively separable in a certain sense that appears difficult to conform to the college application problem. 
\fi

\appendix
\ifen \section{Appendix} \fi

\ifen \subsection{Elementary proof of Theorem \ref{concavityinh}} \fi \label{elementaryconcavityproof}
\begin{proof}
\ifen We will prove the equivalent expression $2 v(\mathcal{X}_h) \geq v(\mathcal{X}_{h+1}) + v(\mathcal{X}_{h-1})$. Applying Theorem \ref{nestedapplication}, we write $\mathcal{X}_h = \mathcal{X}_{h-1} \cup\{j\}$ and $\mathcal{X}_{h+1} = \mathcal{X}_{h-1} \cup\{j, k\}$. If $t_k \leq t_j$, then 
\fi
\begin{align}
\begin{split}
2 v(\mathcal{X}_h) &= v(\mathcal{X}_{h-1} \cup\{j\}) + v(\mathcal{X}_{h-1} \cup\{j\}) \\
&\geq v(\mathcal{X}_{h-1} \cup\{k\}) + v(\mathcal{X}_{h-1} \cup\{j\}) \\
&= v(\mathcal{X}_{h-1} \cup\{k\}) + (1 - f_j) v(\mathcal{X}_{h-1}) + f_j \operatorname{E}[\max\{t_j, X_{h-1}\}] \\
&= v(\mathcal{X}_{h-1} \cup\{k\}) - f_j v(\mathcal{X}_{h-1}) + f_j \operatorname{E}[\max\{t_j, X_{h-1}\}] + v(\mathcal{X}_{h-1})  \\
&\geq v(\mathcal{X}_{h-1} \cup\{k\})  - f_j v(\mathcal{X}_{h-1}\cup\{k\}) + f_j \operatorname{E}[\max\{t_j, X_{h-1}\}]+ v(\mathcal{X}_{h-1})\\
&= (1 - f_j) v(\mathcal{X}_{h-1} \cup\{k\})  + f_j \operatorname{E}[\max\{t_j, X_{h-1}\}]+ v(\mathcal{X}_{h-1})\\
&=  v(\mathcal{X}_{h-1} \cup\{j, k\}) + v(\mathcal{X}_{h-1})\\
&=  v(\mathcal{X}_{h+1}) + v(\mathcal{X}_{h-1}).
\end{split} 
\end{align}
\ifen The first inequality follows from the optimality of $\mathcal{X}_h$, while the second follows from the fact that adding $k$ to $\mathcal{X}_{h-1}$ can only increase its valuation.
\fi

\ifen If $t_k \geq t_j$, then the steps are analogous:
\fi
\begin{align}
\begin{split}
2 v(\mathcal{X}_h) &= v(\mathcal{X}_{h-1} \cup\{j\}) + v(\mathcal{X}_{h-1} \cup\{j\}) \\
&\geq v(\mathcal{X}_{h-1} \cup\{k\}) + v(\mathcal{X}_{h-1} \cup\{j\}) \\
&= (1 - f_k) v(\mathcal{X}_{h-1}) + f_k \operatorname{E}[\max\{t_k, X_{h-1}\}] +  v(\mathcal{X}_{h-1} \cup\{j\})  \\
&= v(\mathcal{X}_{h-1}) - f_k  v(\mathcal{X}_{h-1}) + f_k \operatorname{E}[\max\{t_k, X_{h-1}\}] +  v(\mathcal{X}_{h-1} \cup\{j\})  \\
&\geq v(\mathcal{X}_{h-1}) - f_k  v(\mathcal{X}_{h-1}\cup\{j\}) + f_k \operatorname{E}[\max\{t_k, X_{h-1}\}] +  v(\mathcal{X}_{h-1} \cup\{j\})  \\
&= v(\mathcal{X}_{h-1}) + (1 - f_k) v(\mathcal{X}_{h-1}\cup\{j\}) + f_k \operatorname{E}[\max\{t_k, X_{h-1}\}]  \\
&= v(\mathcal{X}_{h-1}) + v(\mathcal{X}_{h-1} \cup\{j, k\})
\end{split}\\
&= v(\mathcal{X}_{h-1})  + v(\mathcal{X}_{h+1}) \qedhere
\end{align}
\end{proof}

\pagebreak
\ifen
\section*{References}
\addcontentsline{toc}{section}{References}
\fi
\noindent

\parskip 0em
\leftskip 2em
\parindent -2em

Acharya, Mohan S., Asfia Armaan, and Aneeta S. Antony. 2019. ``A Comparison of Regression Models for Prediction of Graduate Admissions.'' In \emph{Second International Conference on Computational Intelligence in Data Science.} \url{https://doi.org/10.1109/ICCIDS.2019.8862140}.

Badanidiyuru, Ashwinkumar and Jan Vondrák. 2014. ``Fast Algorithms for Maximizing Submodular Functions.'' In \emph{Proceedings of the 2014 Annual ACM--SIAM Symposium on Discrete Algorithms}, 1497--1514. \url{https://doi.org/10.1137/1.9781611973402.110}.

Balas, Egon and Eitan Zemel. 1980. ``An Algorithm for Large Zero-One Knapsack Problems.'' \emph{Operations Research} 28 (5): 1130--54. \url{https://doi.org/10.1287/opre.28.5.1130}. 

Bezanson, Jeff, Alan Edelman, Stefan Karpinski, and Viral B. Shah. 2017. ``Julia: A Fresh Approach to Numerical Computing.'' \emph{SIAM Review} 59: 65–98. \url{https://doi.org/10.1137/141000671}.

Blum,  Manuel, Robert W. Floyd, Vaughan Pratt, Ronald L. Rivest, and Robert E. Tarjan. 1973. ``Time Bounds for Selection.'' \emph{Journal of Computer and System Sciences} 7 (4): 448--61. \url{https://doi.org/10.1016/S0022-0000(73)80033-9}.


Calinescu, Gruia, Chandra Chekuri, Martin Pál, and Jan Vondrák. 2011. ``Maximizing a Monotone Submodular Function Subject to a Matroid Constraint.'' \emph{SIAM Journal on Computing} 40 (6): 1740--66. \url{https://doi.org/10.1137/080733991}.

Carraway, Robert, Robert Schmidt, and Lawrence Weatherford. 1993. ``An Algorithm for Maximizing Target Achievement in the Stochastic Knapsack Problem with Normal Returns.'' \emph{Naval Research Logistics} 40 (2): 161--73. \url{https://doi.org/10.1002/nav.3220400203}.

Cormen, Thomas, Charles Leiserson, and Ronald Rivest. 1990. \emph{Introduction to Algorithms.} Cambridge, MA: The MIT Press.

Dantzig, George B. 1957. ``Discrete-Variable Extremum Problems.'' \emph{Operations Research} 5 (2): 266--88.

Dean, Brian, Michel Goemans, and Jan Vondr\'ak. 2008. ``Approximating the Stochastic Knapsack Problem: The Benefit of Adaptivity.'' \emph{Mathematics of Operations Research} 33 (4): 945--64. \url{https://doi.org/10.1287/moor.1080.0330}.

Kellerer, Hans, Ulrich Pferschy, and David Pisinger. 2004. \emph{Knapsack Problems.} Berlin: Springer.

Kulik, Ariel, Hadas Shachnai, and Tami Tamir. 2013. ``Approximations for Monotone and Nonmonotone Submodular Maximization with Knapsack Constraints.'' \emph{Mathematics of Operations Research} 38 (4): 729--39. \url{https://doi.org/10.1287/moor.2013.0592}.

\ifen Jeon, Minhee. 2015. ``[College application strategy] Six chances total\dots divide applications across reach, target, and safety schools'' (in Korean). Jungang Ilbo, Aug. 26. \url{https://www.joongang.co.kr/article/18524069}.\fi

Fisher, Marshall, George Nemhauser, and Laurence Wolsey. 1978. ``An Analysis of Approximations for Maximizing Submodular Set Functions—I.'' \emph{Mathematical Programming} 14: 265--94. 


Fu, Chao. 2014. ``Equilibrium Tuition, Applications, Admissions, and Enrollment in the College Market.'' \emph{Journal of Political Economy} 122 (2): 225--81. \url{https://doi.org/10.1086/675503}. 

Garey, Michael and David Johnson. 1979. \emph{Computers and Intractability: A Guide to the Theory of NP-Completeness.} New York: W. H. Freeman and Company. 

Hartlaub, Vanessa and Thorsten Schneider. 2012. “Educational Choice and Risk Aversion: How Important Is Structural vs. Individual Risk Aversion?” \emph{SOEPpapers on Multidisciplinary Panel Data Research,} no. 433. \url{https://www.diw.de/documents/publikationen/73/diw_01.c.394455.de/diw_sp0433.pdf}.

Kahneman, Daniel. 2011. \emph{Thinking, Fast and Slow.} New York: Macmillan.

Markowitz, Harry. 1952. ``Portfolio Selection.'' \emph{The Journal of Finance} 7 (1): 77--91. \url{https://www.jstor.org/stable/2975974}.

Martello, Silvano and Paolo Toth. 1990. \emph{Knapsack Problems: Algorithms and Computer Implementations.} New York: John Wiley \& Sons. 

Meucci, Attilio. 2005. \emph{Risk and Asset Allocation.} Berlin: Springer-Verlag, 2005. 


Rozanov, Mark and Arie Tamir. 2020. ``The Nestedness Property of the Convex Ordered Median Location Problem on a Tree.'' \emph{Discrete Optimization} 36: 100581. \url{https://doi.org/10.1016/j.disopt.2020.100581}.

Sklarow, Mark. 2018. \emph{State of the Profession 2018: The 10 Trends Reshaping Independent Educational Consulting.} Technical report, Independent Educational Consultants Association. \url{https://www.iecaonline.com/wp-content/uploads/2020/02/IECA-Current-Trends-2018.pdf}.

Sniedovich, Moshe. 1980. ``Preference Order Stochastic Knapsack Problems: Methodological Issues.'' \emph{The Journal of the Operational Research Society} 31 (11): 1025--32. \url{https://www.jstor.org/stable/2581283}. 

Steinberg, E. and M. S. Parks. 1979. ``A Preference Order Dynamic Program for a Knapsack Problem with Stochastic Rewards.'' \emph{The Journal of the Operational Research Society} 30 (2): 141--47. \url{https://www.jstor.org/stable/3009295}. 


Van Huizen, Thomas and Rob Alessie. 2019. ``Risk Aversion and Job Mobility.’’ \emph{Journal of Economic Behavior \& Organization} 164: 91--106. \url{https://doi.org/10.1016/j.jebo.2019.01.021}.

Vazirani, Vijay. 2001. \emph{Approximation Algorithms.} Berlin: Springer. 

Wolsey, Laurence. 1998. \emph{Integer Programming.} New York: John Wiley \& Sons.

\end{document}